\documentclass{amsart}
\usepackage{graphicx}
\usepackage{amssymb,amscd,amsthm,amsxtra}
\usepackage{latexsym}
\usepackage{epsfig,esint}
\usepackage{xcolor}
\newtheorem{thm}{Theorem}[section]
\newtheorem{cor}[thm]{Corollary}
\newtheorem{lem}[thm]{Lemma}
\newtheorem{prop}[thm]{Proposition}
\theoremstyle{definition}
\newtheorem{defn}[thm]{Definition}
\theoremstyle{remark}
\newtheorem{rem}[thm]{Remark}
\numberwithin{equation}{section}

\newcommand{\be}{\begin{equation}}
\newcommand{\ee}{\end{equation}}

\newcommand{\R}{\mathbb R}

\newcommand{\eps}{\varepsilon}

\newcommand{\p}{\partial}
\newcommand{\pa}{\partial}

\newcommand{\comment}[1]{}

\newenvironment{proof1}[1]{\begin{trivlist} \item[] {\em Proof of #1:}}{\newline \textcolor{white}{.}\hfill $\Box$
                      \end{trivlist}}

\begin{document}

\title[Two-phase problems]{Two-phase problems: Perron solutions and regularity  of the Neumann problem in convex cones}
\author{T. Beck}
\address{Mathematics Department, Fordham University, New York, NY }
\email{\tt  tbeck7@fordham.edu}
\author{D. De Silva}
\address{Department of Mathematics, Barnard College, Columbia University, New York, NY 10027}
\email{\tt  desilva@math.columbia.edu}
\author{O. Savin}
\address{Department of Mathematics, Columbia University, New York, NY 10027}\email{\tt  savin@math.columbia.edu}
\begin{abstract} We investigate a fully nonlinear two-phase free boundary problem with a Neumann boundary condition on the boundary of a general convex set $K \subset \R^n$ with corners. We show that the interior regularity theory developed by Caffarelli for the classical two-phase problem in his pioneer works \cite{C1,C2}, can be extended up to the boundary for the Neumann boundary condition under very mild regularity assumptions on the convex domain $K$. To start, we establish a general existence theorem for the Dirichlet two-phase problem driven by two different fully nonlinear operators, which is a result of independent interest.
 \end{abstract}

\maketitle

\section{Introduction}

We investigate a two-phase free boundary problem with a Neumann boundary condition on the boundary of a general convex set $K \subset \R^n$. We focus on the local properties of solutions near points on $\p K$. After a translation and dilation, we will assume throughout that $0 \in \p K$, and $K$ is an open convex set which intersects $\p B_1$:  
\begin{equation}\label{KK} 0 \in \p K,  \quad B_\delta (y) \subset K \quad \quad \mbox{for some} \quad y \in \p B_1, \delta>0.\end{equation}
In particular, $K \cap B_1$ is a Lipschitz domain with Lipschitz constant depending on the parameter $\delta$. We consider the two-phase free boundary problem in $K \cap B_1$ for  fully nonlinear operators $\mathcal F^\pm$ and a nonlinear transmission condition $G$, with a Neumann boundary condition on $\p K$, that is
 \begin{equation}\label{intro}
\left\{
\begin{array}{ll}
\mathcal{F}^+(D^{2}u)=0, & \hbox{in} \quad B_1^+(u):=\{ u>0 \} \cap B_1 \cap K,\\
\  &  \\
\mathcal{F}^-(D^{2}u)=0, & \hbox{in} \quad  B_1^-(u):= \{ u \le 0 \}^\circ \cap B_1 \cap K, \\
\  &  \\
|\nabla u^+| = G(|\nabla u^-|), & \hbox{on} \quad  F(u):= \partial_{\overline K} \{u>0\} \cap B_1,\\
\  &  \\
u_\nu=0, & \hbox{on} \quad \p K \cap B_1 \setminus F(u),\\
\end{array}
\right. 
\end{equation}
where $\p_{\overline K}$ denotes the boundary in $\overline K$.
Here $u \in C(\overline K \cap B_1)$ solves \eqref{intro} in the appropriate viscosity sense (see Section 2).
The operators $ \mathcal F^\pm$ are uniformly elliptic with ellipticity constants $\lambda, \Lambda$ and $ \mathcal  F^\pm(0)=0$, while $G:[0,\infty)\to [0,\infty)$ is a strictly increasing continuous function with $G(t) \to \infty$ as $ t \to \infty$.

The goal of the paper is to show that the interior regularity theory developed by Caffarelli for the classical two-phase problem in his pioneer works \cite{C1,C2}, can be extended up to the boundary for the Neumann boundary condition under very mild regularity assumptions on the convex domain $K$, and for a very general class of problems. The whole point is that we allow domains $K$ with corners - for simplicity, we focus on  conical domains that we decompose in the form $K_{n-m} \times \R^m$ with $K_{n-m} \subset \R^{n-m}$ a convex cone that does not contain a line.  In our main results we obtain estimates for solutions $u$ and their free boundaries which are independent of the smoothness of the cross-section cone $K_{n-m}$. 

For the classical two-phase minimization problem of  Alt, Caffarelli, and Friedman \cite{ACF}, the first author together with Jerison and Raynor \cite{BJR} showed the Lipschitz continuity of
 minimizers up to the fixed boundary under Neumann boundary conditions using an almost monotonicity formula. In addition to convexity, they require
a Dini condition governing the rate at which the fixed boundary converges to its limit
cone at each boundary point. The two dimensional case for the variational problem was already settled in \cite{GMR} for arbitrary convex domains, and so was the purely one-phase minimization problem \cite{R}. We remark that for the Dirichlet problem, Gurevich has shown that the Lipschitz property up to the boundary frequently fails \cite{G}.

Here we consider a much larger class of problems. Before we discuss the behavior of solutions near the fixed boundary $\p K$ where the Neumann condition holds, we first establish a general existence theorem for the standard Dirichlet  two-phase problem, which is a result of independent interest. The existence theory of a Lipschitz solution by Perron's method for operators in divergence form was first obtained by Caffarelli \cite{C3} and extended to the inhomogeneous case by De Silva, Ferrari, and Salsa  \cite{DFS2}. The case of fully nonlinear concave operators was settled by P.Y. Wang \cite{W}  and generalized to the case of equations with non-zero right hand side by Salsa, Tulone, and Verzini \cite{STV}. All these works rely on the powerful Alt-Caffarelli-Friedman (ACF) monotonicity formula \cite{ACF} and on its variants \cite{MP}. In this paper, we avoid the use of the ACF monotonicity formula and establish continuity properties of the Perron's solution based on the ideas in \cite{DS}. In this way we show the existence of viscosity solutions for a general class of two-phase problems involving possibly different operators in the two different phases.

Precisely, let $\mathcal F^\pm$ be fully nonlinear uniformly elliptic operators with the same ellipticity constants and $\mathcal F^\pm(0)=0$,  and let $G$ be as above. Consider the two-phase problem (we refer the reader to the next section for the notion of viscosity solutions),
\begin{equation}
\left\{
\begin{array}{ll}
\mathcal F^+ (D^2 u^+)=0, & \hbox{in} \quad B_1^+(u):=\{ u>0 \} \cap B_1,\\
\  &  \\
\mathcal F^- (D^2 u^-)=0, & \hbox{in} \quad  B_1^-(u):= \{ u \le 0 \}^\circ \cap B_1, \\
\  &  \\
|\nabla u^{+}|=G(|\nabla u^{-}|), &
\hbox{on} \quad  F(u):= \partial B_1^+(u) \cap B_1,\\
\end{array}
\right.  \label{fb}
\end{equation}
subject to Dirichlet boundary conditions $$u = \varphi \quad \mbox{on} \quad \partial B_1,$$
with $\varphi$ a continuous function. Define the family of continuous (viscosity) supersolutions 
$$ \mathcal A:=\left\{v \in C(\overline {B}_1) | \quad \mbox{$v$ is a supersolution}, \quad v \ge \varphi \quad \mbox{on} \quad \p B_1  \right\}.$$
Notice that $\mathcal A \ne \emptyset$ since it contains all large constant functions and $v \ge - \|\varphi\|_{L^\infty}$. 
Then, we define the least supersolution
$$ u(x)= \inf_{v \in \mathcal A} v(x), \quad x \in \overline B_1.$$

Our first main theorem states that the least supersolution solves our two-phase problem.

\begin{thm}\label{main}
$u \in C(\overline {B}_1)$ solves \eqref{fb} in the viscosity sense.
\end{thm}

Going back to our Neumann problem \eqref{intro}, after establishing existence of the Neumann-Dirichlet problem with the strategy developed in Theorem \ref{main}, we proceed to analyze the main question of Lipschitz regularity up to the boundary. For this purpose, we require that $\mathcal F^+=\mathcal F^-$. While we could consider the same class of convex sets as in \cite{BJR}, to fix ideas we focus only on conical domains.  
Recall that any convex cone $K$ can be split (after a rotation) as 
\begin{equation}\label{cki}
K=K_{n-m} \times \R^m,
\end{equation} with $K_{n-m}$ a convex cone in $\R^{n-m}$ satisfying,
$$ K_{n-m} \subset \{x_1 \ge \delta |(x_1,...,x_{n-m})| \} \subset \R^{n-m}, \quad \mbox{for some $\delta>0$.}$$
For $m<n-1$, this ensures that $K_{n-m}$ is strictly included in a half-space, while the case when $m=n-1$ and $K_1=(0,\infty)$ corresponds to the standard half-space situation. 

Our main result is stated below.

 \begin{thm}[Lipschitz regularity] \label{main2}
Assume $u$ solves \eqref{intro} with $K$ satisfying \eqref{KK} and \eqref{cki}, and $$\mathcal F^+=\mathcal F^-, \quad \mbox{and} \quad G'(t) \to 1 \quad \mbox{ as} \quad t \to \infty.$$ If $0 \in \p K\cap F(u)$ then, for $x\in \overline K\cap B_1$, $$|u(x)| \le C (\|u\|_{L^\infty(K\cap B_1)} +1) |x|,$$
with $C$ depending on $n$, $\lambda$, $\Lambda$, $\delta$ and $G$.
\end{thm}

We also refer to Theorem \ref{thm:Lipschitz} for a more general version of Theorem \ref{main2}.

The strategy to prove Theorem \ref{main2} is driven by the results in \cite{DS}. The heuristic is that in the regime of “big gradients” the free boundary condition becomes a continuity (no-jump) condition for the gradient. Then, the Lipschitz estimate follows from the $C^{1,\alpha}$ regularity up to the boundary for fully nonlinear elliptic equations with Neumann boundary condition in conical domains, which we establish here and which constitutes a result of independent interest.

The paper is organized as follows. In Section 2 we present the definition of viscosity solution to \eqref{intro} and prove basic properties, such as stability, and provide examples which support and motivate our definition. In the following section, we provide the proof of Theorem \ref{main} and make some important remarks about generalizations and stronger results. Section 4 deals with the standard Neumann problem for a fully nonlinear operator in a conical domain - a $C^{1,\alpha}$ regularity result up to the boundary is established. In Section 5, we investigate the transmission problem, which occurs as the linearization of our two-phase problem. Finally, in the last section, we present the proof of Theorem \ref{main2}, relying on the estimates for the transmission problem and a linearization technique which has a Harnack inequality as its main tool.

\section{Preliminaries and definitions}

This section contains the definition of viscosity solution for the different types of problems which we will deal with in the rest of the papers. Other basic definitions, some straightforward properties, and some motivating examples are also included. 

\subsection{Preliminaries on Fully Nonlinear Operators} First, we recall some standard facts about fully nonlinear uniformly elliptic operators. For a comprehensive treatment of the topic, we refer the reader to \cite{CC}.

We say that  $\mathcal F$ is a fully nonlinear uniformly elliptic operator 
if there exist $0<\lambda\leq \Lambda$   positive constants such that for every  $M, N\in \mathcal{S}^{n\times n},$  with  $N\geq 0,$
$$
\lambda\| N\|\leq \mathcal F(M+N)-\mathcal F(M)\leq \Lambda \| N\|,
$$
where $\mathcal{S}^{n\times n}$ denotes the set of real  $n\times n$ symmetric matrices. We write $N \geq 0,$ whenever $N$ is non-negative definite. Also, $\|M\|$ denotes the $(L^2,L^2)$-norm of $M$, that is $\|M\|= \sup_{|x|=1} |Mx|$. 

From now on, the class of all fully nonlinear uniformly elliptic operators with ellipticity constants $\lambda,\Lambda$ and such that $\mathcal F(0)=0$ will be denoted by $\mathcal{E}(\lambda, \Lambda).$

We recall the definition of the extremal Pucci operators, $\mathcal M^-_{\lambda,\Lambda}$ and $\mathcal M^+_{\lambda,\Lambda}$: for $0<\lambda \leq \Lambda,$ we set
$$\mathcal{M}^-_{\lambda,\Lambda}(M) = \lambda \sum_{e_i >0} e_i + \Lambda \sum_{e_i <0} e_i,$$
$$\mathcal{M}^+_{\lambda,\Lambda}(M) = \Lambda \sum_{e_i >0} e_i + \lambda \sum_{e_i <0} e_i,$$ with  the $e_i=e_i(M)$ the eigenvalues of $M$.

If $\mathcal{F} \in \mathcal E(\lambda,\Lambda)$ then\begin{equation*}
\mathcal{M}^-_{\frac{\lambda}{n},\Lambda}(M)\leq \mathcal{F}(M)\leq \mathcal{M}^+_{\frac{\lambda}{n},\Lambda}(M).
\end{equation*} 

Finally, it is readily verified  that if $\mathcal{F} \in \mathcal E(\lambda,\Lambda)$ is the rescaling operator defined by
$$
\mathcal{F}_r(M)=\frac{1}{r}\mathcal{F}(rM), \quad r>0
$$
then $\mathcal F_r$ is still an operator in our class $\mathcal E(\lambda,\Lambda).$

 \smallskip
 
\subsection{Definition and properties of viscosity solutions.} We now introduce the definition of viscosity solutions for a uniformly elliptic equation. Here and henceforth we consider operators $\mathcal{F} \in \mathcal E(\lambda,\Lambda)$. First we recall some standard notions.

Given $u, \varphi \in C(\Omega)$, we say that $\varphi$
touches $u$ by below (resp. above) at $x_0 \in \Omega$ if $u(x_0)=
\varphi(x_0),$ and
$$u(x) \geq \varphi(x) \quad (\text{resp. $u(x) \leq
\varphi(x)$}) \quad \text{in a neighborhood $O$ of $x_0$.}$$ If
this inequality is strict in $O \setminus \{x_0\}$, we say that
$\varphi$ touches $u$ strictly by below (resp. above).

\subsubsection{The interior equation.}\label{interior} If $v \in C^2(O)$, with $O$ an open subset in $\R^n,$ satisfies  $$\mathcal F(D^2 v) > 0  \  \ \ (\text{resp}. <0)\quad \text{in $O$,}$$ we call $v$ a (strict) classical subsolution (resp. supersolution) to the equation $\mathcal F(D^2 v) = 0 $ in $O$.

\begin{defn}\label{intdef}Let $\Omega$ be a domain in $\R^n.$ We say that $u \in C(\Omega)$ is a viscosity solution to $$\mathcal F(D^2 u) = 0  \quad \text{in $\Omega$,}$$ if $u$ cannot be touched by  below (resp. above) at an interior point $x_0 \in \Omega$ by a strict classical subsolution (resp. supersolution) in a neighborhood $O$ of $x_0$. 
\end{defn}

The following class of functions plays a fundamental role in the regularity theory for fully nonlinear equations.

\begin{defn}{[The $\mathcal S_\Lambda$ class.]}  We say that
$u \in \underline {\mathcal S}_\Lambda(B_1)$
if $u \in C(B_1)$ is a viscosity subsolution to
$$ \mathcal M_{\frac \lambda n, \Lambda}^+ (D^2 u) \ge 0 \quad \mbox{in} \quad B_1.$$
Similarly, we can define $\overline{\mathcal S}_\Lambda (B_1)$ by considering supersolutions and $\mathcal M_{\frac \lambda n, \Lambda}^-$. Then $$\mathcal S_\Lambda (B_1) = \underline{\mathcal S}_\Lambda (B_1)\cap \overline{\mathcal S}_\Lambda (B_1).$$\end{defn}

\subsubsection{The Neumann Problem.} From now on, whenever we formulate a Neumann problem, we assume that $K$ is an open convex set and satisfies \eqref{KK}. We consider the problem,
\begin{equation}
\left\{
\begin{array}{ll}
\mathcal F(D^2 u)=0, & \hbox{in} \quad  K \cap B_1,\\
\  &  \\
u_\nu=0, & \hbox{on} \quad \p K \cap B_1.\\
\end{array}
\right.  \label{npb0}
\end{equation}

Throughout this paper we denote by $\nu$ the inner unit normal to $K$.

A function  $ u \in C(B_1 \cap \overline K)$ is a viscosity solution to \eqref{npb0} if it satisfies the interior equation according to Definition \ref{intdef} and it satisfies the boundary condition according to the following definition.

\begin{defn}\label{defNR} $ u \in C(B_1 \cap \overline K)$ satisfies
$$u_\nu \ge 0 \quad \text{(resp $\leq 0$)} \quad \mbox{on $\p K$,}$$ if $u$ cannot be touched locally by above (resp. below) at some point $x_0 \in \p K \cap B_1$ by a $C^2$ function $\psi$ that satisfies
$$ \psi_{\nu_0} (x_0) <0, \quad \text{(resp. $>0$)}$$
for all unit directions $\nu_0$ which are normal to a supporting hyperplane at $x_0$ - meaning that $\nu_0 \cdot (x-x_0) >0$ (resp. $<0$) for each $x \in K$.

Notice that the condition $\psi_{\nu_0} (x_0) <0$ for the normal derivatives of $\psi$ at $x_0$ can be stated as 

\smallskip
a) $- \nabla \psi(x_0)$ is interior to the tangent cone $K_{x_0}$ of $K$ at $x_0$.
\end{defn}

Also,  we may restrict to the class of functions $\psi$ that touch $u$ strictly by above at $x_0$ and satisfy 
\smallskip

b) $\mathcal M_{\frac\lambda n ,\Lambda}^+ (D^2 \psi(x_0) )<0$. 

\smallskip

Otherwise, we add to the original $\psi$ a term like $\eps (t + |x-x_0| ^2)- M t^2$ with $t = (x-x_0) \cdot \nu_0$ for some $\nu_0$.

\smallskip

This definition behaves well under limits since the conditions above continue to hold in a neighborhood of $x_0$, and for perturbations of the set $K$ as well. One can easily establish the following stability result. 

\begin{prop}[Stability] \label{stb} Let $u_m$ be viscosity solutions to \eqref{npb0} for $\mathcal F_m \in \mathcal E(\lambda, \Lambda)$ in $K_m\cap B_1$, with $K_m$ satisfying \eqref{KK}. If $$u_m \to u, \quad \p K_m \to \p K, \quad \mathcal F_m \to \mathcal F \quad \mbox{ uniformly on compact sets,}$$
 then $u$ is a viscosity solution to \eqref{npb0} for $\mathcal F$ in $K\cap B_1$.
\end{prop}
\begin{proof} The standard theory of viscosity solutions implies that the equation $\mathcal F(D^2 u)=0$ holds in $K \cap B_1$, see \cite{CC}. Next, we check that $u_\nu=0$ is satisfied on $\p K \cap B_1$. 

Assume by contradiction that $\psi \in C^2_{loc}$ touches $u$ by above at some $x_0 \in \p K \cap B_1$, and $\psi$ satisfies conditions a) and b) in the Definition \ref{defNR} above. Since the graphs of $u_m$ converge to the graph of $u$ in the Hausdorff distance sense, there are constants $\sigma_m \to 0$ such that $\psi+\sigma_m$ touches $u_m$ by above at some point $x_m$, and $x_m \to x_0$. Condition b) guarantees that $x_m$ must be on $\p K_m$. By continuity $-\nabla \psi(x_m)$ is interior to $K_{x_0}$ and therefore interior to $ K_{m,x_m}$ (the tangent cone to $K_m$ at $x_m$) for all large $m$'s, and we contradict that $\p_\nu u_m(x_m) \ge 0$. The case of $\psi$ touching $u$ by below follows analogously.
\end{proof}

\begin{rem}We provide here an example of a viscosity Neumann problem arising from a variational formulation. Let $u$ minimize the $p$-Laplace energy
$$J(u):=\int_{K \cap B_1} |\nabla u|^p dx, \quad \quad p \in (1,\infty),$$
subject to Dirichlet boundary data on $\p B_1 \cap K$. The comparison principle holds with respect to functions $\psi \in C^2$ that have property a) in Definition \ref{defNR} and satisfy
$$ div (\nabla \psi |\nabla \psi|^{p-1}) <0.$$
Indeed, for $v \ge 0$, $v \in W^{1,p}( K \cap B_1)$, $v = 0$ on $\p B_1$, by the convexity of $J$ and by integration by parts we have $$J(\psi +v) - J(\psi) \ge \int_{\p (B_1 \cap K)} v \, \, p |\nabla \psi|^{p-1} (-\nabla \psi) \cdot \nu \,  dx \ge 0,$$
with equality only if $v \equiv 0$. Applying this inequality to $v=(u+\eps-\psi)^+$, with $\eps>0$ sufficiently small to ensure $v=0$ on $\p B_1$ we conclude that:
$$J(u) - J(\min\{u+\eps, \psi\}) = J(\max\{u+\eps, \psi\}) - J(\psi) \geq 0,$$
contradicting that $u$ is minimizer.
\end{rem}

\subsubsection{The two-phase problem.} Let $\mathcal{F}^\pm \in \mathcal E(\lambda, \Lambda).$
Consider the two-phase problem, 

\begin{equation}
\left\{
\begin{array}{ll}
\mathcal{F}^+ (D^2 u)=0, & \hbox{in} \quad B_1^+(u):=\{ u>0 \} \cap B_1,\\
\  &  \\
\mathcal F^- (D^2 u)=0, & \hbox{in} \quad  B_1^-(u):= \{ u \le 0 \}^\circ \cap B_1, \\
\  &  \\
|\nabla u^{+}|=G(|\nabla u^{-}|), &
\hbox{on} \quad  F(u):= \partial B_1^+(u) \cap B_1,\\
\end{array}
\right.  \label{fb1}
\end{equation}
with $G:[0,\infty)\to [0,\infty)$ a strictly increasing continuous function with $G(t) \to \infty$ as $ t \to \infty$. 

A function $u \in C(B_1)$ is a viscosity solution to \eqref{fb1} if the interior equations hold is the sense of Definition \ref{intdef} and the free boundary condition holds in the following sense.

\begin{defn} \label{VSFB} $u\in C(B_1)$ satisfies $$|\nabla u^{+}|=G(|\nabla u^{-}|), \quad \text{on $F(u),$}$$ if at $x_0 \in F(u)$, $u$ cannot be touched locally by functions of the type
$$ a \phi^+ - b \phi^-, \quad \quad \mbox{with} \quad \phi \in C^2, \quad |\nabla \phi(x_0)|=1, \quad a, b \ge 0,$$

1) by below if $a > G(b)$ (supersolution property on $F(u)$); 

2) by above if $ a < G(b)$ (subsolution property on $F(u)$).

\end{defn}

Notice that since $x_0 \in \partial \{u>0\}$ we must have $a>0$ in case 2).

Without loss of generality, we may assume in the definition above that the function $\phi$ satisfies the additional conditions

1')  $\mathcal M_{\frac \lambda n,\Lambda}^-(D^2 \phi(x_0)) >0$ in case 1);

2') $\mathcal M_{\frac \lambda n,\Lambda}^+(D^2 \phi(x_0)) <0$ in case 2). 

This is achieved by adding a large multiple of $\pm d^2$ to $\phi$, where $d$ represents the distance to its zero level set $\{\phi=0\}$, and then modifying slightly the constants $a$ and $b$.

We say that $u \in C(B_1)$ is a supersolution (resp. subsolution) to \eqref{fb} if it is a viscosity supersolution (resp. subsolution) to $\mathcal{F}^\pm $ in $B_1^\pm$, and satisfies the supersolution (resp. subsolution) property on $F(u)$.

It is straightforward to check that the minimum (resp. maximum) of two viscosity supersolutions (resp. subsolutions) is a viscosity supersolution (resp. subsolution).

\subsubsection{The two-phase problem with Neumann Boundary condition.}\label{bo} Consider the two-phase problem in a convex set $K$ with Neumann boundary condition,
 \begin{equation}
\left\{
\begin{array}{ll}
\mathcal{F}^+(D^{2}u)=0, & \hbox{in} \quad B_1^+(u):=\{ u>0 \} \cap B_1 \cap K,\\
\  &  \\
\mathcal{F}^-(D^{2}u)=0, & \hbox{in} \quad  B_1^-(u):= \{ u \le 0 \}^\circ \cap B_1 \cap K, \\
\  &  \\
|\nabla u^{+}|=G(|\nabla u^{-}|), & \hbox{on} \quad  F(u):= \partial_{\overline K} \{u>0\}  \cap B_1,\\
\  &  \\
u_\nu=0, & \hbox{on} \quad \p K \cap B_1 \setminus F(u).\\
\end{array}
\right.  \label{fbn1}
\end{equation}

A function $ u \in C(B_1 \cap \overline K) $ satisfies \eqref{fbn1} in the viscosity sense, if it satisfies the interior equations as in Definition \ref{intdef} and the free boundary condition on $F(u) \cap K$ as in Definition \ref{VSFB}. Finally, we specify the notion of viscosity solutions at points on the boundary $\p K \cap B_1$:

1) $u_\nu=0$ on $\p K \cap \{u>0\}$ and on $\p K \setminus \overline { \{u>0\}}$ as in Definition \ref{defNR};

2) on the set $Z:=\p K \cap \p \{u>0\}$, $|\nabla u^{+}|\ge G(|\nabla u^{-}|)$ means that:

\noindent $u$ cannot be touched locally by above (in $\overline K$) at some point $x_0 \in Z$ by test functions of the type
$$a \psi^+ - b \psi^-, \quad \quad \psi \in C^2, \quad |\nabla \psi(x_0)|=1, \quad a,b \ge 0,$$
with 
$$ a < G(b), \quad \quad \mathcal M_{\frac \lambda n,\Lambda}^+(D^2 \psi(x_0)) <0,$$
and $- \nabla \psi(x_0)$ interior to the tangent cone $K_{x_0}$ of $K$ at $x_0$.

The opposite inequality $|\nabla u^{+}| \le G(|\nabla u^{-}|)$ on $Z$ is defined similarly. 

\begin{rem} To justify further the definition above, we show that it is satisfied for critical points of the $p$-Laplace energy
$$J(u):=\int_{K \cap B_1} (|\nabla u|^p + \chi_{\{u>0\}}) \, dx, \quad \quad p \in (1,\infty),$$
subject to Dirichlet boundary data on $\p B_1 \cap K$, with the transmission function $$G(t)= \left (\frac{1}{p-1}+t^p \right)^ \frac 1p.$$ If a test function $a \psi^+- b \psi^-$ touches a continuous function $u$ by above at $x_0 \in Z$, then we show that $u$ cannot be critical. 

The analysis can be reduced to a one-dimensional computation as follows. After a blow-up of Lipschitz rescalings, we may reduce to the case that $x_0=0$, $K=K_0$ is a convex cone, $\psi:= - \nu_0 \cdot x$ for some unit vector $\nu_0 \in K_0$, and $a<G(b)$. We consider the perturbed test function 
$$\tilde  \Psi:=a \,  \tilde \psi^+ - b \,  \tilde \psi^-,$$
where $$\tilde \psi(x)= d +  4 \eps |d| - \eps d^2,$$ and $d=d(x)$ represents the signed distance from $x$ to the sphere of radius $R+ \eps$ centered at $R \nu_0$ for some $R \gg 1$. 

On the boundary of the annular region $\mathcal U:=\{ |d| <1\}$ we have $$\tilde \psi \ge d + 2 \eps > - \nu_0 \cdot x = \psi \quad \mbox{ hence} \quad  \tilde \Psi \ge u \quad \mbox{on} \quad \p \mathcal U \cap K.$$ 
Inside the annular region $\mathcal U$ we replace $u$ by $ \min\{u, \tilde \Psi\}$ and we claim that the energy of $J$ decreases. This is equivalent to showing that in the domain $ \mathcal U \cap K$ the energy of $\tilde \Psi$ is greater than the energy of $\tilde u:= \max \{u, \tilde \Psi\}$. Since $\tilde \Psi$ is a radial function, it suffices to prove the inequality on each ray that originates from the center $R \nu_0 \in K$. By convexity, each ray intersects $\mathcal U \cap K$ in a segment (an interval) along which $\tilde \Psi$ is monotone increasing, and $\tilde u=\tilde \Psi$ at the left endpoint, while $\tilde u \ge \tilde \Psi$ on the rest of the interval. The one-dimensional energy on each ray is a weighted version of the $J$ functional in an interval $ [-1, \beta]\subset [-1,1]$. Moreover, as we let $R \to \infty$, the weight converges to 1, and the energy becomes the restriction of $J$ in one-dimension. Now the inequality is easy to check as $\tilde \Psi$ minimizes the energy in any interval $[-1,\beta]$ among functions $\tilde u \ge \tilde \Psi$ for which $\tilde u=\tilde \Psi$ at $-1$. 
\end{rem}

The next lemma guarantees that when $G$ is the identity, then it is enough to verify only the Neumann condition on the whole of $\p K \cap B_1$, including $Z$.
\begin{lem}\label{identity}
If $G(t)=t$ then the definition above reduces back to the Definition $\ref{defNR}$ on the whole of $\p K \cap B_1$ (including $Z$).
\end{lem}

\begin{proof}
We first check the subsolution property. Assume by contradiction that $ \psi \in C^2$ touches $u$ strictly by above at a point $x_0\in Z$, with $\mathcal M_{\frac \lambda n,\Lambda}^+(D^2 \psi(x_0))<0$ and $- \nabla \psi(x_0)$ interior to $K_{x_0}$. Consider the perturbed function
$$\tilde \Psi:= (1- \eps) \psi^+ - \psi^- $$
with $\eps$ chosen such that $\tilde \Psi > u$ on $\p B_r(x_0) \cap \overline K$. Notice that $\tilde \Psi$ is strictly monotone in the $\nabla \psi(x_0)$ direction. We translate the graph of $\tilde \Psi$ slightly in the $\nabla \psi(x_0)$ direction until it is above the graph of $\psi$ (and therefore $u$) in $B_r(x_0)$, and then we translate it back continuously until it touches the graph of $u$ by above for the first time in $\overline{B_r}(x_0) \cap \overline K$. Notice that the contact point cannot occur on $\p B_r(x_0)$ or on $\{ u \ne 0\}$. Also it cannot occur on $\p \{u>0\}$ as $\tilde \Psi$ is an admissible test function for comparison (either on $\p K$ or in $K$). Thus the contact point must be interior to $\{u \le 0\}$ but then we contradict either the $\mathcal F^-$ equation if it is in the interior to $K$, or part 1) of the definition at the start of Section \ref{bo} if it is on $\p K$.

For the supersolution property we argue similarly. The situation is slightly simpler since the contact point must belong to $\p \{u>0\}$ directly. 
\end{proof}

Now, we prove a stability result for our definition of viscosity solution of a two-phase problem in a convex set and with Neumann boundary condition. Let $\mathcal F^\pm_m \in \mathcal E(\lambda, \Lambda)$, $G_m:[0,\infty)\to [0,\infty)$ be a family of strictly increasing continuous functions with $G_m(t) \to \infty$ as $ t \to \infty$, and $K_m$ be a family of convex sets satisfying \eqref{KK}.

\begin{lem}[Stability]\label{SS}
Assume $u_m$ solves \eqref{fbn1} in $K_m \cap B_1$, for the operators $\mathcal F^\pm_m$, and the free boundary condition $G_m$.
If $$u_m \to u, \quad \p K_m \to \p K, \quad \mathcal F^\pm_m \to \mathcal F^\pm \quad G_m \to G, \quad \mbox{ uniformly on compact sets,}$$
then $u$ is guaranteed to solve \eqref{fbn1} only partially, in the sense that it satisfies the comparison with all test functions required in the definitions except possibly the ones used in the supersolution property with $b=0$.  
\end{lem}

\begin{proof} The proof is straightforward. If $u$ is touched by above at some point $x_0 \in \p \{u>0\}$ by a test function $a \phi^+-b\phi^-$ with $0<a< G(b)$, we can guarantee the existence of such contact points between perturbations of $u_m$ and of the test function only if $b > 0$. On the other hand, this situation does not occur for the subsolution property when we use test functions that touch the solution by below.\end{proof}

In order to satisfy the supersolution property also with test functions with $b=0$, and obtain the full stability result for the limit it suffices to assume that the $u_m$ are nondegenerate, that is $u_m^+(x) \geq C dist(x, \{u_m=0\})$. 
Another case when the full stability can be inferred is when $G(0)=0$, since then there are no test functions with $b=0$. 

\smallskip

Finally, we deduce the following corollary.

\begin{cor}\label{cormain} If $G(t)=t$, $\mathcal F^\pm=\mathcal F$ in Lemma $\ref{SS}$, then the limiting solution $u$ solves \eqref{npb0}.

\end{cor}

\section{The Perron's method}

In this section we prove our main existence Theorem \ref{main}. We divide the proof into several steps.

\smallskip

\noindent{\textit{Proof of Theorem $\ref{main}$}.}

{\it Step 1.} $u= \varphi$ on $\p B_1$, and it achieves the boundary data continuously.

\

We use as barriers translations of the family of subsolutions (supersolutions) of the form
$$a \phi^+ - b \phi^-, \quad a > G(b), \quad a,b \ge 0, \quad (\mbox{resp.} \quad a \phi^- - b \phi^+, \quad a < G(b),)$$
with
\begin{equation}\label{Fsub}
\phi(x) = \frac 1M r^{M+1}\left(|x|^{-M} -r^{-M} \right).
\end{equation}
Here $M=M(\lambda, \Lambda)$ is sufficiently large so that $\mathcal M^-_{\frac \lambda n,\Lambda} (D^2 \phi)>0$. 

If $\varphi$ is greater on $\p B_1$ than a translation by $z_0 \notin B_1$ of such a subsolution, then any $v \in \mathcal A$ is greater than this translation in $B_1$. This is because $v$ satisfies the comparison principle with the monotone continuous family of subsolutions obtained by translations by $t z_0$, with $t$ decreasing from $\infty$ to $t=1$.  

This shows that $u$ is bounded below by such subsolutions. Choosing $a$, $b$, $r$, $z_0$ accordingly we find $$\liminf_{x \to x_0} u(x) \ge \varphi (x_0), \quad \quad \forall \, \, x_0 \in \p B_1.$$
The other inequality is obtained similarly by working with supersolutions.

\

{\it Step 2:} $u^+$ is continuous, and $\mathcal F^+(D^2 u)=0$ in $\{u>0\}$.

\

We show that $u^+$ is uniformly H\"older continuous on compact sets of $B_1$. Let $x_0 \in B_1 \cap \{u>0\}$, say $x_0=0$ for simplicity, and denote $$u(0)=:\sigma >0.$$ It suffices to show that there exists $\delta>0$ with $\sigma \le C \delta^\alpha$, with $\alpha$ small depending only on $n$, $\lambda, \Lambda$ and $C$ large depending also on $G$ and $\|\varphi\|_{L^\infty}$, such that $v >0$ in $B_\delta$ for any $v \in \mathcal A$. 

Then $u>0$ in $B_\delta$ and $\mathcal F^+(D^2 u)=0$ in $B_\delta$ by the Perron's method for the operator $\mathcal F^+$, see \cite{CIL}. Combining this with the interior estimates for the $\mathcal F^+$ operator we establish the H\"older continuity of $u^+$. 

Assume by contradiction that there exists $v$ such that $ \sigma \le v(0)$ and $B_d \subset \{v>0\}$ is tangent to $F(v)$ with 
\begin{equation}\label{cont0}C d^\alpha < \sigma.
\end{equation} Without loss of generality, we may assume further that $v$ solves 
\begin{equation}\label{l+}
\mathcal F^+ (D^2 v)=0 \quad\mbox{ in} \quad  \{v > \sigma /2\}.
\end{equation} 
This is achieved by replacing $v$ in the open set
$$\Omega_\eps:= \{ x \in B_1| \quad dist(x, \{v \le 0\}) > \eps \},$$ by the solution to the Dirichlet problem
$$ \mathcal F^+ (D^2 w)=0 \quad \mbox{in} \quad \Omega_\eps, \quad w=v \quad \mbox{on} \quad \p \Omega_\eps.$$ 
This replacement does not affect the free boundary $F(v)$.
Notice that $\Omega_\eps$ satisfies the exterior ball condition, and therefore the Dirichlet problem is solvable in the class of continuous functions. We obtain \eqref{l+} by choosing $\eps$ sufficiently small.   

The functions $v^-$ and $(v-\frac \sigma 2)^+$ are subsolutions in $B_1$ (i.e. they belong to $\underline{\mathcal S}_\Lambda (B_1)$) and have disjoint supports. Next we use the weak Harnack inequality (see \cite{CC}) which states that a subsolution $w \in \underline{\mathcal S}_\Lambda(B_1)$ satisfies the diminish of oscillation 
$$ \|w^+\|_{L^\infty (B_{r/2})} \le (1-c(\mu)) \|w^+\|_{L^\infty(B_r)}, \quad  \quad \mbox{if} \quad \frac{|\{w=0\} \cap B_r|}{|B_r|} \ge \mu, $$
with $c(\mu)>0$ depending on $n$, $\lambda, \Lambda$ and $\mu$. Moreover, $c(\mu) \to 1$ as $\mu \to 1$.
We choose $\mu$ accordingly so that in each $B_r$, depending on the density of $\{v \ge 0\}$ in $B_r$, we have either
$$ \|v^-\|_{L^\infty (B_{r/2})} \le \tfrac 18  \|v^-\|_{L^\infty(B_r)},$$
or
$$ \|(v- \tfrac \sigma 2)^+\|_{L^\infty (B_{r/2})} \le (1-c_0)  \|(v- \tfrac \sigma 2)^+\|_{L^\infty(B_r)},$$
with $c_0>0$ small, depending on $n$, $\lambda$, and $\Lambda$.
 
 We apply these inequalities in the dyadic balls $B_r$ with $r=2^{-k}$, $k=0,1,..., N$ with $N$ the last value for which $2^{-N} \ge d$. If for half of the values of of $k \in \{0,1,..,N\}$ we end up in the second alternative then we obtain $$ \frac \sigma 2 \le C' d ^\alpha,$$ with $\alpha$ small depending on $\lambda, \Lambda$, $n$ and $C'$ depending also on $\|\varphi\|_{L^\infty}$,  and we contradict \eqref{cont0}.
 Thus, for at least half the values of $k$ we satisfy the first  alternative, and hence
 $$ \|v^-\|_{L^\infty(B_{2d})} \le C d^{\frac 32},$$
 with $C$ depending on $\| \varphi\|_{L^\infty}$. 
 
 In particular $v \ge - C d^{3/2}$ in $B_{2d}$. This leads to a contradiction if $d$ is too small.
  Indeed, after replacing $v$ in $B_\delta$ with the solution to the Dirichlet problem for the $\mathcal F^+$ operator, we still have $v(0) \ge \sigma$ and so $v \ge c \, \sigma$ in $B_{d/2}$ from the Harnack inequality, for $c=c(n,\Lambda) >0$ small. Then 
 $$v \ge a \phi \quad \mbox{in} \quad B_d \setminus B_{d/2}, \quad a=c \, \frac \sigma d \ge d^{\alpha -1},$$
 with $\phi$ as in \eqref{Fsub} and $r=d$. On the other hand, $v >0$ in $B_d$, $v \ge - C d$ in $B_{2d}$ and the maximum principle imply that 
 $$v \ge b \phi  \quad \mbox{in} \quad B_{2d} \setminus B_{d}, \quad b=C_1 d^{1/2},$$  with $\phi$ as in \eqref{Fsub}, $r=d$, and $C_1$ a fixed constant. This contradicts the free boundary condition for $v$ at the point on $F(v) \cap \p B_\delta$ if $d$ is small, since then $a > G(b)$.

\

{\it Step 3:} $u^-$ is continuous, and $\mathcal F^-(D^2 u)=0$ in $\{u < 0\}$.

\

The proof is similar to the one of Step 2. We sketch some of the details. Assume $$u(0)=-\sigma, \quad \quad \sigma>0,$$ and we show that there exists $v \in \mathcal A$ such that $v<0$ in $B_ \delta$ with $C \delta^\alpha \ge \sigma$, and $\alpha$ small depending only on $n$, $\lambda, \Lambda$ and $C$ large depending also on $G$, $\|\varphi\|_{L^\infty}$ .

Let $v \in \mathcal A$ with $v(0) \le - \sigma /2$, and let $d$ denote the distance from $0$ to $\{v > 0\}$. It suffices to show that if $$C d^\alpha < \sigma,$$ then we can find another element $ \tilde v \in \mathcal A$ with $\tilde v(0) \le - \sigma/2$ and $\tilde v <0$ in $B_{2d}$, and then iterate this property a finite number of times.

After modifying $v$ in its positive phase as in Step 2, we may assume that $$\mathcal F^+(D^2 v)=0 \quad \mbox{ in} \quad  \{v > \eps\} \cap B_1,$$ for some $\eps>0$ small, to be made precise later. Then $v^-$ and $(v-\eps)^+$ are subsolutions with disjoint supports, and from the weak Harnack inequality either
$$ \|(v- \eps)^+\|_{L^\infty (B_{r/2})} \le \frac 18  \|(v-\eps)^+\|_{L^\infty(B_r)},$$
or
$$ \|v^-\|_{L^\infty (B_{r/2})} \le (1-c_0)  \|v^-\|_{L^\infty(B_r)}.$$
As above this implies that $(v-\eps)^+ \le C d^{3/2}$ in $B_{4d}$, thus 
\begin{equation}\label{up3}
v \le 2C d^\frac 32 \quad \mbox{ in} \quad  B_{4d}
\end{equation} 
by choosing $\eps$ sufficiently small. We can replace $v$ by the solution to $\mathcal F^-(D^2 w)=0$ in $B_d$ and obtain that
\begin{equation}\label{up31}
v \le - c \sigma \quad \mbox{in} \quad B_{d/2}.
\end{equation}
In the annulus $B_{4d} \setminus B_{d/2}$ we consider $\tilde v$ to be the minimum between $v$ and the explicit function
$$a \phi^- - b \phi^+, \quad a = C_1 d ^ \frac 12, \quad b = c_1 \frac \sigma d \ge d ^{\alpha -1},$$
with $\phi$ as in \eqref{Fsub} and $r=2d$. The constants $C_1$, $c_1$ are chosen such that the function above is greater than $v$ on $\p (B_{4d} \setminus B_{d/2})$, see \eqref{up3}-\eqref{up31}. If $d$ is small, then $a < G(b)$, and then $\tilde v \in \mathcal A$ (we extend $\tilde v$ by $v$ outside the annulus  $B_{4d} \setminus B_{d/2}$.)  

In conclusion, $u^-$ is uniformly H\"older continuous in compact sets of $B_1$, and clearly $\mathcal{F}^-(D^2 u)=0$ in $\{u<0\}$. 

\

{\it Step 4:} $u \in \mathcal A$, and $\mathcal F^-(D^2 u)=0$ in $\{u \leq 0\}^\circ$.

\

Clearly $u$ is a supersolution to the $\mathcal F^-$ equation in the set $\{u \leq 0\}^\circ$. 

Next we check the supersolution property for $u$ on the free boundary $F(u)$. 

Assume by contradiction that the graph of a function $a\phi^+ -b \phi^-$ that satisfies 1) in Definition \ref{VSFB} touches $u$ by below at $x_0 \in F(u)$. After modifying the function $\phi$ slightly we may assume further that $a$, $b >0$, 1') is satisfied, and that the test function touches $u$ strictly by below at $x_0$. This implies that if $v_n \in \mathcal A$ is a sequence such that $v_n(x_0) \to u(x_0)$, then small translations of the test function in the direction of $- \nabla \phi(x_0)$ touch the graphs of $v_n$'s at interior points converging to $x_0$. By 1') the contact point must belong to $F(v_n)$ which contradicts that $v_n \in \mathcal A$ is a supersolution. 

Since $u$ is the least supersolution for $\mathcal F^-$ in any ball included in $\{u < 0\}^\circ$ it follows that $\mathcal F^-(D^2 u)=0$ in $\{u < 0\}^\circ$. Therefore, either $u \equiv 0$ or $u<0$ in the set $\{u \le 0\}^\circ$.

\

{\it Step 5:} $u$ is a subsolution.

\

Assume by contradiction that the graph of a function $a\phi^+ -b \phi^-$ that satisfies 2), 2') in the Definition \ref{VSFB} touches $u$ by above at $x_0 \in F(u)$ in $B_r(x_0)$. We perturb $\phi$ in $B_r(x_0)$ as $$\tilde \phi:= \phi + \eps (|x-x_0|^2 - r^2),$$
with $\eps$ small, and let $$v:= \min \{ u, a \tilde \phi^+ -b \tilde \phi^-\} \quad \mbox{ in} \quad  B_r(x_0),$$ and extend $v$ to equal $u$ outside $B_r(x_0)$. Then it follows that $v \in \mathcal A$, hence $u \le v$. On the other hand, since $v \le 0$ in a neighborhood of $x_0$ we contradict that $x_0 \in \p \{u>0\}$.

\qed

\begin{rem}

i) The proof can be applied to the greatest subsolution as well, and it gives that it solves \eqref{fb} in the viscosity sense.

\smallskip

ii) If $G(0)>0$ then the least supersolution constructed has the additional property that $u^+$ is nondegenerate i.e. 
$$ u^+(x) \ge c \, \, dist(x, \{u=0\}), \quad \quad \forall x \in B_{1/2},$$
with $c$ depending on $n$, $\lambda, \Lambda$, $G(0)$. 

This follows from the fact that if $B_r(x_0) \subset B_1 \cap \{u>0\}$, then we must have $u(x_0) \ge c r$. Otherwise, $u \le c' r$ in $B_{r/2}(x_0)$ for some small $c'$, and we can construct a lower supersolution $v= \min\{ u, \phi \}$, with $\phi$ a supersolution as in \eqref{Fsub} with $b=0$, which vanishes in $B_{r/4}(x_0)$.

\smallskip

iii) The proof shows that $u \in C_{loc}^\alpha(B_1)$ for some $\alpha$ that depends only on $n, \lambda, $ and $\Lambda$. On the other hand, the $C^\alpha$ norm of $u$ in $B_{1/2}$ depends also on 
$\|\phi\|_{L^\infty}$ and the value of $G(1)$. 

\smallskip

iv) The Lipschitz continuity of $u$ can be obtained if 

a) $\mathcal F^+$ and $\mathcal F^-$ are bounded above by the same constant coefficient linear operator (then one can apply the ACF monotonicity formula)

b) $G(t)/t \to 1$ and $\mathcal F^+=\mathcal F^-$ as in \cite{DS}. 

\end{rem}

\

Finally, we observe that our proof applies also to more general nonlinear operators $\mathcal F(D^2u,Du,u)$ provided they satisfy the two key properties:

a) the Harnack inequality; 

b)  large multiples of radial barriers as in \eqref{Fsub} continue to be subsolutions. 

An example of such an operator is the $p$-Laplace operator.

\section{The Neumann problem in convex cones} \label{sec:Neumann}

In this section we study the mixed Dirichlet-Neumann problem
\begin{equation}\begin{cases}\label{Ne}
\mathcal F(D^2 u)=0 \quad \mbox{in} \quad K \cap B_1,\\
$$ u = \varphi \quad \mbox{on} \quad \p B_1 \cap \overline K, \quad \quad u_\nu =0 \quad \mbox{on} \quad \p K \cap B_1,\end{cases}\end{equation} with $K$  an open convex set satisfying \eqref{KK} and $\mathcal F \in \mathcal E(\lambda, \Lambda)$.

Existence and uniqueness are established in Proposition \ref{EXU} below and they follow after a modification of the known techniques for the Dirichlet problem (see \cite{CC,CIL}).
The main result of the section is a pointwise $C^{1,\alpha}$ regularity theorem for solutions of \eqref{Ne} in a conical domain $K$, which plays an essential role towards the proof of our main Theorem \ref{main2}.

Recall that any convex cone $K$ can be split (after a rotation) as 
\begin{equation}\label{ck}
K=K_{n-m} \times \R^m,
\end{equation} with $K_{n-m}$ a convex cone in $\R^{n-m}$ satisfying
 $$ K_{n-m} \subset \{x_1 \ge \delta |(x_1,...,x_{n-m})| \} \subset \R^{n-m}, \quad \mbox{for some $\delta>0$.}$$
When $m<n-1$, $K_{n-m}$ is strictly included in a half-space, and the case when $m=n-1$ and $K_1=(0,\infty)$ corresponds to the standard half-space situation. Here we are interested in estimates that depend on the parameter $\delta$ but are independent of the smoothness of the cross-section cone $K_{n-m}$.

\begin{thm}\label{T08} 
Let $u\in C(\overline K \cap B_1)$ be a viscosity solution to \eqref{Ne} and $K$ a convex cone as in \eqref{ck}. Then, for all $x$ in $\overline K \cap B_1$,
$$ |u(x) - u(0) - \tau_m \cdot x | \le C \|u \|_{L^\infty} |x|^{1+\alpha},$$ with $\tau_m$ a vector in $\{0\} \times \R^m$ and $\alpha \in (0,1).$
\end{thm}

The constants $C$ and $\alpha$ in Theorem \ref{T08} depend only on $n$, $\lambda$, $\Lambda$ and $\delta$, and the same for the constants appearing in the steps of the proof below.

We start by introducing the standard classes of functions which are subsolutions/supersolutions to linear equations with measurable coefficients and satisfy the Neumann condition on $\p K$. 

\begin{defn}{[The $\mathcal S_\Lambda$ class.]} \label{sla2} We say that
$$u \in \underline {\mathcal S}_\Lambda(K \cap B_1)$$
if $u \in C( \overline K \cap B_1)$ is a viscosity subsolution to
$$ \mathcal M_{\frac \lambda n, \Lambda}^+ (D^2 u) \ge 0 \quad \mbox{in} \quad K \cap B_1, \quad  u_\nu \ge 0 \quad \mbox{on} \quad \p K \cap B_1,$$
in the sense of Definitions \ref{intdef}-\ref{defNR}. Similarly, we can define $\overline{\mathcal S}_\Lambda (K\cap B_1)$ by considering supersolutions. Then $$\mathcal S_\Lambda (K\cap B_1) = \underline{\mathcal S}_\Lambda (K\cap B_1)\cap \overline{\mathcal S}_\Lambda (K\cap B_1).$$\end{defn}

We state a version of maximum principle in this setting.

\begin{lem}[Maximum principle] \label{lem:max}
Assume that $u \in \underline {\mathcal S}_\Lambda(K \cap B_1)$ is continuous on $\p B_1$. If $u \le 0$ on $\p B_1$, then $u \le 0$ in $\overline {K \cap B_1}$.

\end{lem}

\begin{proof} Let $ B_{\delta'}(y) \subset K \setminus \overline{B_1}$, with $\delta'$ small depending on $\delta$ from \eqref{KK}, and consider the continuous family of test functions
\begin{equation}\label{psit}
\psi_t(x):= t((\delta')^{-M}-|x-y|^{-M}), \quad \quad t >0,
\end{equation}
with $M$ sufficiently large (depending on $\lambda,\Lambda ,n$) such that $ \mathcal M_{\frac \lambda n, \Lambda}^+ (D^2 \psi_t) < 0$. 

Notice that $-\nabla \psi_t(x)$ points in the direction of $y - x$, hence $\psi_t$ cannot touch $u$ by above on $\p K \cap B_1$. Also $\psi_t$ cannot touch $u$ by above in the interior of $K\cap B_1$, or on the boundary $\p B_1$ since here $\psi_t >0 \ge u$.  We deduce that $u < \psi_t$ holds in $\overline{K\cap B_1}$ for all $t>0$. Indeed, this holds for one large value $t=t_0$, and then the inequality is preserved as we can decrease $t$ continuously up to $0$. Taking the limit $t\to0^+$ gives the conclusion. 
\end{proof}

We now establish that solutions are $C^\alpha$ in the interior, all the way up to the boundary of $K$. By standard arguments, it is enough to prove the following.

\begin{lem}\label{cal}
Let $u \in {\mathcal S}_\Lambda(K \cap B_1)$. Then 
$$ osc_{B_{1/2} \cap K} \, u \le (1- c) \, \, osc_{B_1 \cap K} \, u.$$
\end{lem}

\begin{proof} Assume that $$0 \le u \le 1,$$ and for simplicity we prove the inequality above with $B_{1/2}$ replaced by $B_{1/4}$. 

Pick an interior ball, $B_{\delta'/4} (y) \subset K \cap B_{1/8}$ with $\delta'$ small, depending on $\delta$, see \eqref{KK}. Assume that $u(y)$ is closer to the top constraint, i.e. $u(y) \ge 1/2.$ Then, by the interior Harnack inequality $u \ge c$ in $B_{\delta/8}(y)$, and we can compare $u$ in the annulus $B_{1/2}(y) \setminus B_{\delta/8}(y)$ with an explicit barrier of the form $$\psi(x) = c'\left(|x-y|^{-M} -  2^{M}\right),$$
with $c'$ small and $M$ large. The constant $M$ is chosen such that the barrier is a subsolution to the maximal Pucci operator $\mathcal M_{\frac \lambda n, \Lambda}^+$, and $c'$ so that it is below $u$ on $\p B_{\delta/8}(y)$. Notice that the Neumann condition is satisfied since the gradient of the barrier at $x \in \p K$ points in the direction of $y-x$ which is interior to the tangent cone $K_x$. Therefore, $u\geq \psi$ in $B_{1/2}(y) \setminus B_{\delta/8}(y)$, and we get the desired conclusion since the barrier is positive in $B_{1/4}$.
\end{proof}

In order to prove existence and uniqueness, we introduce the following regularization.

\begin{defn}{[The sup-convolution.]}\label{supconv} Let $u$ be a continuous function in $\overline{K \cap B_1}$. Given $\eps>0$, for all $x \in \overline{K \cap B_1}$ we define 
$$u^\eps(x)= \max_{y \in \overline {K \cap B_1}} \left(u(y) - \frac{1}{2 \eps} |x-y|^2 \right).$$\end{defn}
Notice that the point $y=y(x)$ where the maximum is realized satisfies $|y-x| \le C \eps^{1/2}$, where $C$ depends on $\|u\|_{L^\infty}(K\cap B_1)$.
\begin{prop}\label{reg} If $u$ is a subsolution to \eqref{Ne}, that is 
$$\mathcal F(D^2 u)\ge 0 \quad \mbox{in} \quad K \cap B_1, \quad \quad  u_\nu \ge 0 \quad \mbox{on} \quad \p K \cap B_1,$$
then $u^\eps$ is a semiconvex subsolution to \eqref{Ne} in $K \cap B_{1- C \eps ^{ 1/2}}$, where $C$ depends on $\|u\|_{L^\infty}(K\cap B_1)$.
\end{prop}

\begin{proof}An important observation is that if $x \in K \cap B_{1- C \eps ^{ 1/2}}$ then the corresponding $y$ where the maximum is realized must be in $K \cap B_1$ and not on $\p K$. Otherwise, $u(z)$ is touched by above at $y \in \p K \cap B_1$ by 
$$ \psi(z)=u^\eps(x) + \frac{1}{2 \eps} |x-z|^2$$
and we get a contradiction since $-\nabla \psi(y)$ points in the direction $x-y \in K_y$.
Therefore $y$ is an interior point of $K$. It follows that $\mathcal F(D^2 u^\eps(x)) \ge 0$ since if a test function $\varphi\in C^2$ touches $u^\eps$ by above at $x$, then a translation of $\varphi$ touches $u$ by above at $y$.

Next we check the Neumann boundary condition for $u^\eps$. If $x_0 \in \p K$, then $u^\eps$ has a tangent polynomial 
$$p(x)=u(y_0) - \frac{1}{2 \eps} |x-y_0|^2, \quad \quad y_0 \in \overline K,$$
by below at $x_0$, and  notice that $\nabla p(x_0)$ points in the direction of $y_0-x_0 \in \overline {K_{x_0}}$. If $\psi\in C^2$ is a test function that touches $u^\eps$ (and therefore $p$) by above at $x_0$ and satisfies property a) in Definition \ref{defNR}, then
$-\nabla (\psi - p) (x_0)=-\nabla \psi(x_0) + \nabla p(x_0)$ points in the interior of $K_{x_0}$ and we contradict that $\psi-p \ge 0$ along the direction of $-\nabla (\psi - p) (x_0)$.
\end{proof}

We can now establish the following main proposition.

\begin{prop}\label{comp} If $u$ is a subsolution to \eqref{Ne} and $v$ is a supersolution to \eqref{Ne}, then $u-v \in \underline {\mathcal S}_\Lambda$. 

\end{prop}

\begin{proof} First we check that $u^\eps-v^\eps\in \underline {\mathcal S}_\Lambda$. Indeed, $u^\eps-v^\eps$ is a subsolution in the interior by the previous lemma. Also, as in the proof above, at a boundary point $x_0 \in \p K \cap B_1$, $u^\eps-v^\eps$ has a quadratic tangent polynomial touching by below at $x_0$ of the form
$$ a - \frac 1 \eps |x-z_0|^2, \quad z_0 \in \overline K,$$
which shows the Neumann condition holds on $\p K$. The conclusion follows by letting $\eps \to 0$ in $u^\eps-v^\eps\in \underline {\mathcal S}_\Lambda$.
\end{proof}
 
 Next we obtain the existence and uniqueness for the Neumann problem in $K$.
  
 \begin{prop} \label{EXU}Assume that $K$ intersects $\p B_1$ transversally. Then the Dirichlet-Neumann problem \eqref{Ne} has a unique solution.
 \end{prop}
 Here we say that $K$ intersects $\p B_1$ transversally if at any intersection point $x_0 \in \p K \cap \p B_1$ the tangent cone $K_{x_0}$ of $K$  is not included in the tangent cone of $B_1$ at $x_0$. In particular in any small ball $B_r(x_0)$ and $r >0$ small, there exists a ball of comparable size included in $K \setminus B_1$, i.e. $B_{\mu r}(y) \subset (K \setminus B_1) \cap B_r(x_0)$ with $\mu$ a small constant depending only on $K$. 
 
 \begin{proof}
The uniqueness is a consequence of Proposition \ref{comp} above, and the maximum principle of Lemma \ref{lem:max}.

The existence follows by Perron's method. We sketch some of the details. 

We define 
$$ \mathcal A:=\left\{v \in C(\overline {K \cap B_1}) | \quad \mbox{$v$ is a subsolution}, \quad v \le \varphi \quad \mbox{on} \quad \p B_1 \cap \overline K \right\},$$
and show that 
\begin{equation}\label{usup}
u = \sup_{v \in \mathcal A} v,
\end{equation}
solves \eqref{Ne}. By the remark above, it follows that all points in $\p B_1 \cap \overline K$ are regular for the Dirichlet problem in the sense that for each $x \in \p B_1 \cap \overline K$ there exists a supersolution $\psi_x$ which vanishes at $x$ and is positive everywhere else in $\overline {B_1 \cap K}$. The function $\psi_x$ is obtained as the infimum over suitable family of explicit barriers as in \eqref{psit}. These functions give the existence of continuous barriers $\underline \varphi \le \overline \varphi$, with
$\underline \varphi \in \underline {\mathcal S}_\Lambda$ a subsolution and $ \overline \varphi \in \overline {\mathcal S}_\Lambda$ a supersolution defined in $\overline{B_1 \cap K}$ which agree with $\varphi$ on $\p B_1 \cap \overline K$.
After replacing $v \in \mathcal A$ by $$\tilde v= \max \{v, \underline \varphi\} \in \mathcal A, \quad \quad \tilde v \le \overline \varphi,$$ we may assume that all $v$ in \eqref{usup} have a continuous modulus of continuity $\rho(r)$ on $\p B_1 \cap \overline K $. Furthermore, we may assume that $\rho$ is smooth outside the origin, and then we replace each $v$ as above by its $\rho$-regularization
$$ \bar v(x):= \max_{y \in \overline{K \cap B_1}} v(y)- \rho(|x-y|), \quad \quad x \in \overline{K \cap B_1}.$$
As in the proof of Proposition \ref{reg} one can check that $\bar v \in \mathcal A$ and $\bar v= \varphi$ on $\p B_1 \cap \overline K$. Moreover, $\bar v \ge v$ and $\bar v$ has a fixed modulus of continuity in $\overline{K \cap B_1}$. By Arzela-Ascoli theorem, we find that $u$ in \eqref{usup} is the uniform limit of a sequence of such $v_k$'s and therefore $u \in \mathcal A$ as well. Now it is straightforward to verify that $u$ is a supersolution as well.
 \end{proof}
 
 We investigate further regularity of a solution $u$ of \eqref{Ne}. Next we remark that the weak Harnack inequality holds for functions in the $S_\Lambda$ classes.
 
 \begin{prop}[Harnack inequality] \label{prop:Harnack} Assume $u \ge 0$.

a) There exists $\eps>0$ universal such that if $u \in \overline {\mathcal S}_\Lambda(K \cap B_1)$ then
$$\left(\int_{K \cap B_1} u^ \eps dx\right)^\frac 1 \eps \le C \inf_{ \overline K \cap B_{1/2}} u.$$
b) If $u \in \underline {\mathcal S}_\Lambda(K \cap B_1)$ and $p>0$ then
$$\sup_{ \overline K \cap B_{1/2}} u \le C(p) \left(\int_{K \cap B_1} u^ p dx\right)^\frac 1 p \quad \mbox{in} \quad \overline K \cap B_{1/2}.$$
c) If $u \in {\mathcal S}_\Lambda(K \cap B_1)$ then
$$u(x) \le Cu(y) \quad \mbox{for all} \quad x,y \in \overline K \cap B_{1/2}.$$
\end{prop}

The constants $C$ above depend additionally on $\lambda, \Lambda, n, \delta$.
\begin{proof}We only sketch the proof. It suffices to establish part a) since the implications $a) \Rightarrow b) \Rightarrow c)$ follow by standard arguments as in \cite{CC}. 

Since $K \cap B_1$ is a connected Lipschitz domain, the interior weak Harnack inequality gives the desired $L^\eps$ bound provided the infimum in the right hand side is taken over an interior ball $B_{\delta'}(y)$, with $B_{2\delta'}(y) \subset K \cap B_1$. It remains to show that
$$ \inf_{B_{\delta'}(y)} u \le C \inf_{ \overline K \cap B_{1/2}} u.$$
We argue similarly as in Lemma \ref{cal} and compare $u$ with a multiple of the radial subsolution 
$$ \psi(x):= |x-y|^{-M} - (3/4)^{-M}, \quad \quad \mathcal M_{\frac \lambda n, \Lambda}^-(D^2 \psi) >0,$$
in $(B_{3/4}(y) \cap K) \setminus B_{\delta'}(y)$. This gives the desired inequality by choosing $y$ sufficiently close to the origin.
\end{proof}

 The next lemma contains the key observation towards the $C^{1,\alpha}$ regularity result Theorem \ref{T08}. First, notice that by the interior $C^{1,\alpha}$-regularity estimates, if $u$ is a solution to \eqref{Ne}, then $$w:=|\nabla u|^2$$ is well defined in the interior. We extend $w$ to $\p K$ by upper semicontinuity, i.e.
 $$ w(x_0) = \limsup_{x \in K, \, x \to x_0} w(x), \quad x_0 \in \p K \cap B_1.$$ 
 
 \begin{lem} \label{nabu}Let $u$ be a viscosity solution to \eqref{Ne}. Then, $$w:=|\nabla u|^2 \in \underline {\mathcal S}_\Lambda(K\cap B_1).$$

 \end{lem}
 
 In Lemma \ref{nabu} the definition of the class $\underline {\mathcal S}_\Lambda$ is extended to the class of upper semicontinuous functions. This is a standard generalization and the previous results concerning $\underline {\mathcal S}_\Lambda$ carry through to this setting as well.  
 
\begin{proof} Since $u \in C^{1,\alpha}_{loc}(K \cap B_1)$, we find that $$ \frac{u^\eps - u}{\eps} \quad \to  \quad \frac 12 |\nabla u|^2=w, \quad \quad \quad \mbox{as $\eps \to 0$,}$$
 where $u^\eps$ is the sup-convolution as in Definition \ref{supconv}. This convergence is uniform only on compact sets of $K \cap B_1$, which means that $w$ is a subsolution in the interior of $K \cap B_1$. It remains to check the viscosity definition for $w$ at points on $\p K$. However it is not yet clear if $w$ is bounded near $\p K$.

To establish the boundedness of $w$ in $K \cap B_{1/2}$ we use the following two facts about the quantities that appear in the convergence above:

a) if $|\nabla u|^2 \le M$ in $B_{2r}(x) \cap K$, with $r \ge C\eps^{1/2},$ then $$ u^{\eps} - u \le \frac M 2 \, \eps \quad \mbox {in} \quad B_r(x)\cap K,$$

b) conversely, if this inequality is satisfied for all $\eps$ small, then $|\nabla u|^2 \le M$ in $B_r(x)\cap K$.

Part a) is a consequence of $u(y)-u(x) \le M^{\frac 12}|x-y|$ and part b) of the pointwise convergence in the beginning of the proof.

The interior gradient estimates for the equation $\mathcal F(D^2u)=0$ (see \cite{CC}) imply that $$|\nabla u(x)| \le C \|u\|_{L^\infty(K\cap B_1} d(x)^{-1}, \quad \mbox{in} \quad B_{3/4} \cap K,$$ where $d(x)$ is the distance from $x$ to the boundary $\p K$. By a) above
$$\frac{u^\eps - u}{\eps} \le C \left(\min\{d(x), \eps^{1/2}\}\right )^{-2},$$
and the right hand side has bounded $L^p$ norm for some small fixed $p>0$. Since the left hand side is in $\underline {\mathcal S}_\Lambda$ (see Proposition \ref{comp}), the Weak Harnack inequality Proposition \ref{prop:Harnack} implies that it is uniformly bounded in $K \cap B_{1/2}$. Thus, $w$ is uniformly bounded in $K \cap B_{1/2}$ by part b) above. 

Finally, assume by contradiction that a test function $\psi$ as in Definition \ref{sla2} touches $w$ strictly by above at some boundary point $x_0 \in \p K$. Then by properties a) and b) above it follows for sufficiently small $\eps$ that $\psi+ const.$ touches $2(u^\eps - u)/\eps$ by above at some point $x_{\eps}$ close to $x_0$, and we contradict Proposition \ref{comp}.\end{proof}

Next we show that the gradient $\|\nabla u\|_{L^\infty}$ must decay when restricting to smaller scales provided that the convex set $K$ has a modulus of strict convexity at $0$. In order to quantify the modulus of strict convexity, we consider the collection of all inner unit normals $\nu(x)$ to all supporting planes to $K$ at the points $x \in \p K \cap B_{1/2}$ and require that they are not contained in a set of zero measure on the unit sphere $\p B_1$. Another equivalent way is to say that there exists a supporting plane to $K$ at some point in $\p K \cap B_{1/2}$ which strictly separates from $\p K \cap (B_1 \setminus B_{1/2})$. 

To make this notion precise, we assume that for some small constant $\delta>0$, 
\begin{equation}\label{Stc}
B_\delta(\nu(x_0))  \cap \p B_1\quad \subset \quad  \{\nu(x),\, x \in \p K \cap B_{1/2}\} \quad \subset \quad \p B_1,
\end{equation}
for some point $x_0 \in \p K \cap B_{1/2}$. 

\begin{lem}\label{L06} Let $u$ be a solution to \eqref{Ne}. Assume, in addition to \eqref{KK}, that the strict convexity condition \eqref{Stc} holds. Then 
$$\|\nabla u\|_{L^\infty(K\cap B_{1/2})} \le (1-c(\delta)) \|\nabla u\|_{L^\infty(K \cap B_{1})}.$$

\end{lem}
\begin{proof}
Assume that 
\begin{equation}
\label{gradn}
\|\nabla u\|_{L^\infty(K\cap B_1)}=1,
\end{equation} 
and let $z$ be an interior point to $K \cap B_1$, such that $$B_\delta(z) \subset K \cap B_1.$$  

We claim that $|\nabla u(z)| \le 1- \mu$ for some $\mu>0$ small, depending on $\delta$ and the universal constants $\lambda$, $\Lambda$, $n$.

 Then by the interior $C^{1,\alpha}$ estimates we have $$\|\nabla u\|_{L^\infty(B_{c(\mu)}(z))} \le 1 - \frac \mu 2,$$ for some $c(\mu)>0$ small. The weak Harnack inequality part a) applied to $$1-|\nabla u|^2 \ge 0$$ which belongs to $\overline {\mathcal S}_\Lambda(K \cap B_1)$ (see Lemma \ref{nabu}) gives the desired bound
$$ |\nabla u| \le 1- c'(\mu)  \quad \mbox{in} \quad K \cap B_{1/2}.$$

To establish the claim, assume by contradiction that $$|\nabla u(z)| \ge 1-\mu,$$ with $\mu$ sufficiently small. Let $e$ denote the unit direction of $\nabla u(z)$. Then
$$ u_e \le 1 \text{ in }K\cap B_1, \quad u_e(z) \ge 1- \mu,$$
and the interior Harnack inequality applied to $u_e$ gives that
$$|u_e -1| \le C(\delta,\sigma) \mu \quad \mbox{in} \quad K \cap B_{7/8} \cap \{ d_{\p K}(x) > \sigma \},$$
where $d_{\p K}(x)$ denotes the distance function from $x$ to $\p K$.
This and \eqref{gradn}, imply that if $\sigma$ and $\mu$ are chosen sufficiently small (depending on a parameter $\eps = \eps(\delta)$ to be specified later), then $u$ is well approximated by a linear function of slope $e$, 
\begin{equation}\label{e1}
|u - u(0)-e \cdot x| \le \eps \quad \mbox {in} \quad B_{7/8} \cap K.
\end{equation}
On the other hand, our hypothesis \eqref{Stc} implies that there exists $\rho>0$ small depending on $\delta$, and  $x_1 \in \p K \cap B_{3/4}$ such that
$$ \mbox{either} \quad \nu \cdot e > c(\delta) \quad \mbox{or} \quad \nu \cdot e <- c(\delta) \quad \forall x \in \p K \cap B_\rho(x_1).$$
The proof of this fact follows easily by compactness. Indeed, if this statement fails for a sequence of sets $K_m$ with $\rho_m \to 0$, $c_m(\delta) \to 0$, then (up to a subsequence) the $K_m$'s converge to a limiting set $K_\infty$ whose boundary (in $B_{3/4}$) contains line segments in the $e$ direction. Then $K_m$ cannot satisfy property \eqref{Stc} for $m$ large.

Let us assume for simplicity that the second alternative above holds. Then, by \eqref{e1}, the polynomial $P(x) + const.$ with
$$P(x)= u(0) + e \cdot x + C(\rho) \eps \left(|x-x_1|^2 + (x-x_1) \cdot \nu_1 - M [ ( x-x_1) \cdot \nu_1]^2 \right),$$
touches $u$ by above at some point in $\overline K \cap B_\rho(x_1)$. The contact point must be on $\p K$ and then, if $\eps$ is chosen sufficiently small, depending on $c(\delta)$,
$$ \nu(x) \cdot \nabla P = \nu(x) \cdot (e + O(\eps)) <0,$$
and we reach a contradiction.
\end{proof}

We remark that in the proof above we showed that either the $L^\infty$ norm of $\nabla u$ decays when restricting from $B_1$ to $B_{1/2}$, or that both $u$ and $K$ are well approximated by a linear function and a cylindrical convex set respectively.  

In the case when $K$ is a cone with vertex at $0$, we can iterate Lemma \ref{L06} indefinitely and obtain the following result.

\begin{cor} Assume that $u$ solves \eqref{Ne} and that $K$ is a convex cone included in a half-space i.e.
$$K \subset \{x_1 \ge \delta |(x_1,..,x_n)| \},$$
for some $\delta>0$. Then $$ |\nabla u(x)| \le C \|u\|_{L^\infty} |x|^\alpha, \quad \quad \forall \, x \in K \cap B_{1/2},$$
with $C$, $\alpha$ constants depending on $n$, $\lambda$, $\Lambda$, $\delta$.
\end{cor}

Indeed, the bound 
\begin{equation}\label{gradest}
\|\nabla u\|_{L^\infty(K \cap B_{1/2})} \le C \|u \|_{L^\infty(K \cap B_1)}
\end{equation} follows by Lemma \ref{nabu} and then iterating Lemma \ref{L06} in dyadic balls.

Next we consider the setting of Theorem \ref{T08} in which $K$ is a general cone which is split as  $K_{n-m} \times \R^m$ with $K_{n-m} \subset \R^{n-m}$ strictly included in a half-space, see \eqref{ck}. We state a version of Lemma \ref{L06}. from which our main Theorem \ref{T08} follows via standard arguments.

\begin{lem}\label{L07}Assume that $u$ solves \eqref{Ne} and that  $K=K_{n-m}\times \R^m$ as in \eqref{ck}. Then,
$$\|\nabla u - \tau_m\|_{L^\infty(K \cap B_{1/2})} \le (1-c(\delta)) \|\nabla u\|_{L^\infty(K \cap B_{1})}.$$
with $\tau_m$ a vector in $\{0\} \times \R^m$.
\end{lem}

\begin{proof} The proof follows from the arguments of Lemma \ref{L06}: 

If $\|\nabla u\|_{L^\infty(K \cap B_1)} =1$ then, as in the proof of Lemma \ref{L06}, either $$ \|\nabla u\|_{L^\infty(K \cap B_{1/2})}  \le 1 - c'(\mu),$$ 
in which case we are done by taking $\tau_m=0$, or
$$ |u - u(0)-e \cdot x| \le \eps \quad \mbox {in} \quad B_{7/8} \cap K,$$
for some $\eps$ sufficiently small. The proof of Lemma \ref{L06} shows that we may take $ e = \tau_m \in \{0\} \times \R^m$ to be tangential to $\p K$ (after relabeling $\eps$ if necessary). Note that for $\tau \in \{0\}\times\R^m$, then $u- u(0)- \tau \cdot x$ solves the same problem \eqref{Ne}. Therefore, we can apply the gradient bound
$$\|\nabla u - \tau_m\|_{L^\infty(K \cap B_{1/2})} \le C \|u - u(0)-\tau_m \cdot x\|_{L^\infty(K \cap B_{7/8})} \le C' \eps \le 1/2,$$
which gives the conclusion.
\end{proof}

{\it Proof of Theorem \ref{T08}.} We apply \eqref{gradest} and then iterate Lemma \ref{L07} in dyadic balls to find that 
$$|\nabla u(x) - \nabla u(0)| \le C \|u\|_{L^\infty(K \cap B_1)} \, |x|^\alpha, \quad \quad \nabla u \in \{0\} \times \R^m,$$
which gives the desired conclusion.
\qed

\

We conclude this section with the following observation. Using the method from Section 3 and the local solvability of the Dirichlet-Neumann problem, Proposition \ref{EXU}, we can now extend Perron's method to the the two-phase Dirichlet-Neumann problem problem in a convex set $K$ (see \eqref{fbn1}): 

\begin{equation}
\left\{
\begin{array}{ll}
\mathcal F^+(D^{2}u)=0, & \hbox{in} \quad B_1^+(u):=\{ u>0 \} \cap B_1 \cap K,\\
\  &  \\
\mathcal F^-(D^{2}u)=0, & \hbox{in} \quad  B_1^-(u):= \{ u \le 0 \}^\circ \cap B_1 \cap K, \\
\  &  \\
|\nabla u^{+}|=G(|\nabla u^{-}|), & \hbox{on} \quad  F(u):= \partial_{\overline K} \{u>0\} \cap B_1,\\
\  &  \\
u_\nu=0, & \hbox{on} \quad \p K \cap B_1 \setminus F(u),\\
\end{array}
\right.  \label{fbn}
\end{equation}
and $u= \varphi$ on $\p B_1 \cap \overline K$. 

\begin{prop} Assume that $K$ intersects $\p B_1$ transversally. Then the two-phase Dirichlet-Neumann problem \eqref{fbn} has a viscosity solution. 
\end{prop}

The proof is essentially identical to the one of Theorem \ref{main} in Section 3. We only have to remark that we can treat points in $\p K \cap B_1$ as interior points. This is because 
the test functions/barriers of the form $a \phi^+ - b \phi^-$ that were used in the proof remain admissible due to the viscosity formulation of the Neumann boundary condition on $\p K$. We leave the details to the reader.

\section {The Transmission-Neumann problem in convex cones} 
In this section we investigate a transmission problem which arises in the linearization of our two-phase problem with Neumann boundary condition in a convex set $K$ satisfying \eqref{KK}. 
We assume that $K = K' \times \R$ with $K' \in \R^{n-1}$ a convex set and consider the problem, \begin{equation}\label{tran}\begin{cases}
\mathcal F^+(D^2 u) =0 \quad \text{in $K\cap B_1\cap \{x_n >0\}$}\\
\mathcal F^-(D^2 u) =0 \quad \text{in $K \cap B_1\cap \{x_n <0\}$}\\
u_n^+ + \gamma \, u_n^-=0 \quad \mbox{on $\overline K \cap \{x_n=0\}$, $\gamma > 0$,}\\
u_\nu=0 \quad \mbox{on $\p K \cap B_1 \cap \{x_n \neq 0\},$}
 \end{cases}\end{equation} where $u_n^{+}$ (resp. $u_n^-$) denotes the derivative of $u$ at points on $\{x_n=0\}$ in the $e_n$ (resp. $-e_n$) direction, with $u$ restricted to $\{x_n >0\}$ (resp. $\{x_n <0\}$.)
 
 We assume that $\gamma$ is bounded away from $0$ and infinity, say $\gamma \in [\delta, \delta^{-1}]$.

The interior regularity for the transmission problem was studied in \cite{DFS1}, where the authors established the $C^{1,\alpha}$ regularity of solutions in the tangential directions. In this section we extend their results up to the boundary $\p K$ for the Neumann problem. 
  
 Next we specify the viscosity definition of the transmission condition.

 \begin{defn}\label{def5.1}A function $u \in C(\overline K \cap B_1)$ is a viscosity solution to \eqref{tran} if 
 
 1) it satisfies the interior equations in the sense of Definition \ref{intdef}; 
 
 2) the Neumann condition holds in the sense of Definition \ref{defNR};
 
 3) $u$ cannot be touched by above (resp. below) at $x_0 \in K \cap \{x_n=0\}$ by a piecewise $C^2$ function $\psi$ that satisfies 
\begin{equation}\label{psi1}
\psi_n^+(x_0) + \gamma \, \psi^-_n(x_0) <0 \quad \text{(resp. $>0$)}.
\end{equation} 
Here again $\psi_n^{+}$ and $\psi_n^-$ denote the derivatives in the $e_n$ and $-e_n$ directions of $\psi$ restricted to $\{x_n >0\}$ and $\{x_n <0\}$.
As usual, we may assume that at the touching point 
\begin{equation}\label{psi2}
\mathcal M_{\frac \lambda n,\Lambda}^+(D^2 \psi^\pm) <0 \quad \mbox{ (resp. } \quad \mathcal M_{\frac \lambda n,\Lambda}^-(D^2 \psi^\pm) >0. \quad )
\end{equation}

4) at points $x_0 \in \p K' \times \{x_n=0\}$, $u$ cannot be touched by above (resp. below) by a piecewise $C^2$ function $\psi$ that satisfies 
\eqref{psi1}-\eqref{psi2} and
$\psi^\pm_\nu (x_0)<0$ (resp. $>0$)  in the sense of Definition \ref{defNR}.

\end{defn}

Similarly one can only define the notions of subsolution/supersolution for the problem \eqref{tran}.

 We remark that since $K$ is invariant in the $e_n$ direction, the condition $\psi^\pm_\nu (x_0)<0$ at $x_0 = (x_0',0)$ in 4) is the same as the Neumann condition for the restriction of $\psi$ to $\R^{n-1} \times \{0\}$, $$\nabla_{\nu'} \psi|_{\R^{n-1} \times \{0\}}  <0 \quad \mbox{at} \quad x_0' \in \p K'.$$ 
  It is not difficult to check that, for the touching by below property, it is sufficient to restrict the collection of test functions $\psi$ in 4) to include only quadratic polynomials of the form
  $$ A + p x_n^+ + q x_n^- + B \, Q (x-y), \quad y=(y',0) \in K',$$
 with
 $$ Q = \frac{n \Lambda}{\lambda} x_n^2 - |x'|^2, \quad \mbox{so that} \quad \mathcal M^-_{\frac \lambda n, \Lambda} (D^2 Q) >0,$$
and $$B<0 \quad \mbox{(resp. $>0$)} , \quad \quad p + \gamma q <0\quad \mbox{(resp. $>0$).}$$
We can also restrict to the analogous collection of test functions when checking for touching by above. The main result of this section is the following pointwise $C^{1,\alpha}$ estimate which is the analogue of Theorem \ref{T08} in the setting of the transmission problem \eqref{tran}.
As before, in addition to \eqref{KK}, we assume that $K$ is a cone of the form 
\begin{equation}\label{ck2}
K=K_{n-m} \times \R^m, \quad \quad m \ge 1,
\end{equation} with $K_{n-m}$ a convex cone in $\R^{n-m}$ satisfying $$ K_{n-m} \subset \{x_1 \ge \delta |(x_1,...,x_{n-m})| \} \subset \R^{n-m}, \quad \mbox{for some $\delta>0$.}$$

\begin{thm}\label{T09}
Assume that $K$ is a cone as above and that $u$ is a viscosity solution for the transmission problem \eqref{tran} in $\overline K \cap B_1$. If $\|u \|_{L^\infty(K\cap B_1)} \le 1$, then 
$$ |u(x) - u(0) - (px_n^+ + q x_n^- + \tau_m \cdot x) | \le C |x|^{1+\alpha},$$
with $\tau_m \in \{0\} \times \R^{m-1} \times \{0\}$ and $C$, $\alpha$ constants that depend on $\delta$, $n$, $\Lambda$, $\lambda$.
\end{thm}

The strategy to prove the Theorem above, mimics the one we followed in Section 4 for the Neumann problem. The main difference is that the problem in the interior of $K$ is no longer invariant under translations in the $e_n$ direction. Throughout the proof the constants appearing depend on $\delta$, $n$, $\Lambda$, $\lambda$.

In the first part of this section we only assume that $K = K' \times \R$ for some convex set $K'$ (not necessarily a cone), $0 \in \p K$ and $K \cap B_1$ contains a ball of radius $\delta$.

In this case we define the regularization $u^\eps$ as the sup-convolution of translations of $u$ in the first $n-1$ variables. We denote points
$$x=(x',x_n), \quad x' \in \R^{n-1},$$
and let  
\begin{equation}\label{uep}
u^\eps(x):= \max_{(y',x_n) \in \overline K \cap B_1} \left(u(y',x_n) - \frac{1}{2 \eps} |x'-y'|^2 \right), \quad \quad x \in \overline K \cap B_1.
\end{equation}
The inf-convolution $u_{\eps}$ is defined analogously.

Thus the graph of $u^\eps$ over $\overline K \cap B_1$ is obtained as the maximum over translations of the graph of $u$ by the family of vectors
$$ (z',0,- \frac{1}{2 \eps} |z'|^2) \quad \in \R^{n+1}, \quad \quad z' \in \R^{n-1}. $$
Below we will show that the maximum in the definition of $u^\eps(x)$ at any interior point $x \in K \cap B_1$ is achieved at an interior point $(y',x_n) \in K \cap B_1$. This implies that $u^\eps$ is a subsolution for \eqref{tran} in the interior of $K\cap B_1$,  since $u^\eps$ at $x$ has a translation of the graph of $u$ tangent by below (with $z'=x'-y'$ in the vector above). In fact, we show that $u^\eps$ is a subsolution up to the boundary.

\begin{lem}\label{L5.3}
If $u$ is a subsolution to \eqref{tran} with $\|u\|_{L^\infty(K\cap B_1)} \le 1$, then $u^\eps$ is a subsolution to \eqref{tran} in $\overline K \cap B_{1-C \sqrt \eps}$.
\end{lem}

Before we proceed with the proof of Lemma \ref{L5.3} we need a result concerning the behavior of solutions to linear equations in half-cylindrical domains of the type $K' \times [q, \infty)$ with $K' \subset \R^{n-1}$, which satisfy Dirichlet boundary conditions on the slice $K' \times \{q\}$ and Neumann boundary condition on the remaining part of the boundary $\p K' \times (q, \infty)$. Essentially we show that if the normal derivative of the boundary data at some point $z_0 \in \p K' \times \{q\}$ is negative, then the derivative of $u$ at $z_0$ along the $e_n$ direction is $- \infty$. 

Assume that $z_0=0$ for simplicity, and denote by $B_r^+=B_r \cap \{x_n >0\}$.

\begin{lem}\label{L5.4}
Assume that $u \in \underline S_{\Lambda}(K \cap B_1^+)$ is touched by above at $0$ by a quadratic function of the form
$$ \mu_1 |x'|^2 + a_1 x_n - x' \cdot x_0 ',$$
with $x_0'$ an interior point to the tangent cone of $K'$ at $0$, and for some constants $\mu_1$, $a_1$.
Then in $K\cap B_\rho^+$, for some small ball $B_\rho^+$ 
$$u \le \mu_2 |x'|^2 + a_2 x_n - x' \cdot x_0',$$
with $a_2 \le a_1 - \sigma$ for a constant $\sigma>0$ depending only on universal constants and the distance from $x_0'$ to the boundary of the tangent cone of $K'$ at $0$.
\end{lem}

Lemma \ref{L5.4} can be iterated indefinitely and we obtain that in sufficiently small balls, $u$ is bounded above by quadratic polynomials with arbitrarily large negative slope in the $e_n$ direction.

\begin{proof} After a homogeneous of degree 1 dilation we may assume that  $K$ is close to its tangent cone at $0$ and $\mu_1 \ll 1$ is sufficiently small, hence the upper bound is well approximated by the linear function
$$\ell(x):=a_1 x_n - x' \cdot x_0 ',$$
in $B_1$. First we claim that on the slice $x_n = \eta$, with $\eta$ a small universal constant to be specified later, we have the inequality
\begin{equation}\label{s5cl}
u \le \ell - \sigma \quad \mbox{on} \quad \{x_n = \eta\} \cap K \cap B_{1/2},
\end{equation}
for some $\sigma >0$ small.

For this we use the barrier 
$$ \psi:=\ell- \sigma (w-1) $$
with $w \in C^2$ in the domain $ B_1^+ \setminus B_{\delta} (y) \supset K \cap B_1^+$ so that 

1) $\mathcal M_{\frac\lambda n, \Lambda}^-(w) >0$; 

2) $w =0$ on $\p B^+_1$, $w =C$  on $\p B_\delta(y)$ with $C$ sufficiently large so that

3) $w \ge 2$ on $\{x_n = \eta\} \cap K \cap B^+_{1/2}$.

We compare $u$ and $\psi$ in $K \cap B_1^+$.
Notice that $u \le \ell +\sigma \le \psi$ on $\p(K \cap B_1^+) \setminus \p K$. On the remaining  portion of the boundary, $\p K \cap B^+_1$, we have
$$\nabla \psi=(-x_0',a_1) - \sigma \nabla w,$$ hence $-\nabla \psi$ points inside $K$ provided that $\sigma$ is chosen small (recall  that $x_0'$ is an interior direction for $K'$ at $0$.) 
We can conclude that $u \le \psi$ in the whole domain $K \cap B_1^+$, which gives the claim \eqref{s5cl}.

Next we extend the inequality \eqref{s5cl} in a neighborhood of the origin. For this we compare $u$ in the cylinder $(K' \cap B_{1/2}') \times [0, \eta]$ with the barrier function
$$ \phi:= \ell + \sigma (- x_n - n \Lambda x_n^2 + |x'|^2).$$
Notice that $\phi$  is a supersolution, and on the boundary of the domain we have

1) $\phi \ge \ell + \frac{\sigma}{8} \ge u$ when $x' \in \p B_{1/2}'$,  

2) $\phi \ge \ell + \sigma |x'|^2 \ge u$ on $x_n=0$, (since $\mu_1 \ll 1$)

3) $\phi \ge \ell - \frac{\sigma}{2} \ge u$ on $x_n= \eta$, (by \eqref{s5cl}), provided that $\eta$ is chosen small, universal.

We conclude that $u \le \phi$ in the interior of the cylinder, which gives the desired result.

\end{proof}

\begin{proof}[Proof of Lemma \ref{L5.3}]
We first show that $u^\eps$ is a subsolution in $K \cap B_{1- C \sqrt \eps}$. As remarked before the statement of the lemma, it suffices to prove the following claim.

\

{\it Claim:} If $x \in K \cap B_{1-C \sqrt \eps}$, the maximum in \eqref{uep} is realized at a point $y' \in K'$.

\

Assume by contradiction that the maximum at $x_0 \in K$ is realized at some $y_0=(y', e_n \cdot x_0)$ with $y' \in \p K'$, and assume for simplicity that $y'=0$. Then on the slice $(x-x_0) \cdot e_n=0$ the subsolution $u$ is touched by above at the origin (after a vertical translation) by the quadratic polynomial
$$ \frac {1}{2 \eps} |x'|^2 - \frac 1 \eps x' \cdot x_0'.$$
This inequality can be extended near $y_0$ to each side of the hyperplane $(x-x_0) \cdot e_n=0$ to an inequality of the type
\begin{equation}\label{ulef}
u \le \frac {1}{\eps} |x'|^2 - \frac 1 \eps x' \cdot x_0' + M|t|, \quad \quad t:=(x-x_0) \cdot e_n ,
\end{equation}
with $M\gg 1$ large. 
Indeed, for this we compare $u$ in a neighborhood of $y_0$, on each side of the hyperplane $t=0$, with the supersolution
$$\psi:= \frac {1}{\eps} |x'|^2 - \frac 1 \eps x' \cdot x_0' +  M (|t| - t^2).$$
Here $M$ is chosen large so that $\psi \ge u$ on $K \cap \p B_\rho(y_0)$, and notice that on the part of the boundary included in $\p K$, $\psi$ is an admissible test function for the Neumann boundary condition.  
 
Now we can apply Lemma \ref{L5.4}  successively and deduce that the coefficient of $|t|$ in \eqref{ulef} can be made arbitrarily negative as we focus near $y_0$. This implies in particular that $u$ is touched by above at $y_0$ by the piecewise $C^2$ polynomial
$$  \mu |x'|^2 - M t^2 - |t| - \frac 1 \eps x' \cdot x_0'.$$
This contradicts Definition \ref{def5.1} (4) if $y_0 \in \{x_n=0\}$ or Definition \ref{def5.1} (2) if $y_0 \notin \{x_n=0\}$, and the claim is proved.

\smallskip

It remains to show that $u^\eps$ satisfies the subsolution property  at a point $x_0 \in \p K$. By construction, at such a point $\nabla _{x'}u^\eps(x_0)$ points in the direction $y'-x_0'$ for some $y' \in \overline K'$. Then $u^\eps$ satisfies the conditions of Definition \ref{def5.1} at $x_0$ since it cannot be touched by above by a test function $\psi$ which is $C^2$ in the $x'$ variable for which $-\nabla _{x'} \psi(x_0)$ points to the interior of the tangent cone $K'_{x_0'}$.

\end{proof}

We introduce the corresponding $\mathcal S_\Lambda$ classes in the transmission setting which we denote by $\mathcal S_\Lambda^{tr}$.

\begin{defn}
We say that $u \in \underline {\mathcal S}_\Lambda^{tr} (K \cap B_1)$ if $u$ satisfies Definition \ref{def5.1} with operators $\mathcal F^+=\mathcal F^-=\mathcal M^+_{\frac{\lambda}{n},\Lambda}$ and with \eqref{psi1} replaced by
$$\max \left\{\psi_n^+(x_0) + \delta \, \psi^-_n(x_0), \psi_n^+(x_0) + \delta^{-1} \, \psi^-_n(x_0)\right\} <0.$$
Similarly, we define $\overline {\mathcal S}_\Lambda^{tr} (K \cap B_1)$ and  ${\mathcal S}_\Lambda^{tr} (K \cap B_1)$.
\end{defn}

A consequence of Lemma \ref{L5.3} is that the difference between a subsolution and a supersolution of \eqref{tran} belongs to the class $\underline {\mathcal S}_\Lambda^{tr} $.

\begin{prop}
If $u$ is a subsolution and $v$ is a supersolution of \eqref{tran}, then $u-v \in \underline {\mathcal S}^{tr}_\Lambda(K \cap B_1)$. 
\end{prop}

\begin{proof}
Let $u^\eps$ be the sup-convolution of $u$ and and $v_\eps$ the inf-convolution of $v$. It suffices to show that 
$$u^\eps - v_\eps \in \underline {\mathcal S}^{tr}_\Lambda(K \cap B_{1-C \sqrt \eps}),$$ and then get the result follows by letting $\eps \to 0$. The fact that $u^\eps - v_\eps$ belongs to $\mathcal S_\Lambda^{tr}$ in the interior of $K \cap B_{1-C \sqrt \eps}$ was proved in \cite{DFS1}, since $u^\eps$ remains a subsolution and $v_\eps$ a supersolution of \eqref{tran} in the interior of $K$.

The boundary conditions for $u^\eps - v_\eps$ at some point $x_0 \in \p K$ are easily satisfied as in the last part of the proof of Lemma \ref{L5.3} above. Indeed, $\nabla_{x'}(u^\eps-v_\eps)$ still points in the direction $y'-x_0'$ for some $y' \in \overline K'$, hence $u^\eps - v_\eps$ cannot be touched by above by a test function $\psi$ for which $-\nabla _{x'} \psi(x_0)$ points to the interior of $K'_{x_0'}$.
\end{proof}

\begin{lem}The Harnack inequality of Proposition \ref{prop:Harnack} holds for the classes  $\mathcal S^{tr}_\Lambda$ in the transmission problem.
\end{lem}

\begin{proof} As in the proof of Proposition \ref{prop:Harnack} it suffices to show that if $u \in \overline S_\Lambda^{tr}$ and $B_{2\delta'}(y) \subset K \cap B_1$, then 
$$ \inf_{B_{\delta'}(y)} u \le C \inf_{ \overline K \cap B_{1/2}} u.$$
Also, we may assume that $y \in \{x_n=0\}$, since the inequality can be extended to other points in the interior of $K$ by using barriers with compact support in $K \cap B_1$. 

We argue similarly as in Section 4 and consider the admissible test function 
$$ \psi(x):= |x-y|^{-M} - (3/4)^{-M} + \eps |x_n|, \quad \quad \mathcal M_{\frac \lambda n, \Lambda}^-(D^2 \psi) >0,$$
with $\eps$ small universal. We compare $u$ with a multiple of $\psi$ in $\{\psi\ge 0\} \cap K) \setminus B_{\delta'}(y)$ and obtain the desired inequality.
\end{proof}

In \cite{DFS1} it was shown that in the interior of $K$ the derivatives of $u$ in the $x'$ directions are H\"older continuous. Thus,
$$w:= |\nabla_{x'} u|^2$$
is well defined in $K \cap B_1$, and can be extended up to the boundary by upper semicontinuity.

 We have the analogous result to Lemma \ref{nabu} in the transmission setting of \eqref{tran}.
 
\begin{lem}\label{nabu2} Let $u$ be a solution to \eqref{tran}. Then
 $$w:=|\nabla_{x'} u|^2 \in \underline {\mathcal S}^{tr}_\Lambda(K \cap B_1),$$
 and
 $$\|w\|_{L^\infty(K \cap B_{1/2})} \le C \|u\|_{L^\infty(K \cap B_1)}^2.$$
\end{lem}

 The proof is identical to that of Lemma \ref{nabu} and we omit the details.
 
 Next we state the analogue of Lemma \ref{L07} in the setting of the transmission problem \eqref{tran}.
 
 \begin{lem}\label{L072}Assume that $K=K_{n-m}\times \R^m$ as in \eqref{ck2}. Then 
$$\|\nabla_{x'} u - \tau_m\|_{L^\infty(K \cap B_{1/2})} \le (1-c(\delta)) \|\nabla_{x'} u\|_{L^\infty(K \cap B_{1})}.$$
with $\tau_m$ a vector in ${\color{blue}\{0\}\times} \R^{m-1} \times \{0\}$.
\end{lem}
 
 Theorem \ref{T09} follows easily by applying Lemma \ref{nabu2}, then by iterating Lemma \ref{L072} and combining it with following result.
 
 \begin{lem} \label{L510} Assume that $u$ solves \eqref{tran} and $\|\nabla_{x'} u\|_{L^\infty(K\cap B_1)} \le 1$. Then, there exist constants $p$, $q$ such that
 $$|u- (u(0) + px_n^+ + q x_n^-)| \le C \quad \mbox{in} \quad \overline K \cap B_{1/2},$$
 and $p + \gamma q =0$.
  \end{lem}
 
 \begin{proof}[Proof of Lemma \ref{L510}] The hypothesis implies that $|u(x) - h(x_n)| \le 1$ for some one-dimensional function $h(x_n)$. Then, the viscosity definition in the interior of $K$ implies that
 $$h(x_n)= h(0) + px_n^+ + q x_n^- + O(1),$$
 for some constants $p$, $q$ with $p + \gamma q =0$. Indeed, otherwise we can touch $u$ by below/above at some interior point in $B'_\delta(y') \times [- \frac 12, \frac 12]$ with appropriate test functions of the form
 $$ h(0) + px_n^+ + q x_n^- \pm C |x'-y'|^2,$$
 where $y'\in K'$ with $B'_\delta(y') \subset K' \cap B_1'$.
 \end{proof}
 We are left to prove Lemma \ref{L072}.
 
 \begin{proof}[Proof of Lemma \ref{L072}]
The proof is similar to  that of Lemma \ref{L07}, and we only sketch some of the details. Notice that \eqref{tran} remains invariant under the addition of piecewise linear functions of the form 
$$ p x_n^+ + q x_n^- + \tau_m \cdot x, \quad \quad  \quad p+\gamma q=0, \quad \tau_m \in \{0\} \times \R^{m-1} \times \{0\}.$$
Assume $\|\nabla_{x'} u\|_{L^\infty(K \cap B_1)} =1$. Then, as in Lemma \ref{L06}, either $$ \|\nabla_{x'} u\|_{L^\infty(K \cap B_{1/2})}  \le 1 - c'(\mu),$$ 
in which case we are done by taking $\tau_m=0$, or by the interior Harnack inequality applied to a directional derivative $e' \cdot \nabla u \in S_\Lambda^{tr}$, we have
$$ |u(x) - u(0,x_n) - e' \cdot x'| \le \eps, $$
for some unit direction $e' \in \R^{n-1} \times \{0\}$ and $\eps$ sufficiently small. As in Lemma \ref{L510}, this implies that in $K \cap B_{7/8}$
$$  |u(x) - u(0) - (p x_n^+ + q  x_n^- + e' \cdot x')| \le C \eps, \quad p + \gamma q =0.$$
Furthermore, the uniform strict convexity of $K_{n-m}$ and the argument in the proof of Lemma \ref{L06} shows that we may take $ e' = \tau_m \in \{0\} \times \R^{m-1}\times \{0\}$ to be tangential to $\p K_{n-m}$ (after relabeling $\eps$ if necessary).
Now we apply the gradient bound of Lemma \ref{nabu2} for the difference between $u$ and its piecewise approximation and obtain
$$\|\nabla u - \tau_m\|_{L^\infty(K \cap B_{1/2})} \le C \eps \le 1/2,$$
which gives the conclusion.
\end{proof}

\section{The two-phase Neumann problem in convex cones}

In this section, $u$ solves the two-phase Neumann problem in $K \cap B_1$ given by
 \begin{equation}\label{eqn:solution}
\left\{
\begin{array}{ll}
\mathcal{F}^+(D^{2}u)=0, & \hbox{in} \quad B_1^+(u):=\{ u>0 \} \cap B_1 \cap K,\\
\  &  \\
\mathcal{F}^-(D^{2}u)=0, & \hbox{in} \quad  B_1^-(u):= \{ u \le 0 \}^\circ \cap B_1 \cap K, \\
\  &  \\
|\nabla u^+| = G(|\nabla u^-|), & \hbox{on} \quad  F(u):= \partial_{\overline K} \{u>0\}  \cap B_1,\\
\  &  \\
u_\nu=0, & \hbox{on} \quad \p K  \cap B_1 \setminus F(u),\\
\end{array}
\right. 
\end{equation}
understood in the viscosity sense of Subsection \ref{bo}.
In this section, in addition to \eqref{KK}, we assume that $K$ is a convex cone of the form $K = K_{n-m}\times \R^m$, for a cone $K_{n-m}$ included in a half-space,
 \begin{equation}\label{conco}
 K_{n-m} \subset \{x_1 \ge \delta |(x_1,...,x_{n-m})| \} \subset \R^{n-m}, \quad \mbox{for some $\delta>0$.}
 \end{equation}

Our main purpose here is to prove the Lipschitz regularity of solutions to \eqref{eqn:solution}, which in particular establishes Theorem \ref{main2}. 

\begin{thm}[Lipschitz regularity] \label{thm:Lipschitz}
Let $u$ be a solution to \eqref{eqn:solution} with $K$ a cone as in \eqref{conco}. Assume that $\mathcal F ^+=\mathcal F^-$, and
\begin{equation}\label{eps0}
G'(t)\in [1- \eps_0, 1+\eps_0], \quad\mbox{ for all large} \quad t \ge t_0,
\end{equation} 
with $\eps_0$ small depending on $n$, $\lambda,\Lambda$ and $\delta$. If $0 \in F(u)$, then $$|u(x)| \le C (\|u\|_{L^\infty(K\cap B_1)} +1) |x|,$$ for  all $x\in \overline K\cap B_1$,
with $C$ a constant that depends only on $n$, $\lambda,\Lambda$, $\delta$ and $t_0$.
\end{thm}
As mentioned in the Introduction, to prove this theorem, we follow closely the proof of Theorem 1.1 in \cite{DS}, adapting where necessary to obtain estimates up to the boundary of the cone $K$. The proofs here only sketch the arguments, and we refer the reader to \cite{DS} for the precise details.

 The first step is to establish the H\"older regularity of solutions.

\begin{lem} \label{lem:Holder} Assume that $u$ satisfies  \eqref{eqn:solution}. Then, $u\in C^{0,\alpha}(B_{1/2}\cap K)$ for some $\alpha>0$ that depends on $n$, $\lambda,\Lambda$, $\delta$, and
$$ \|u\|_{C^{0,\alpha}(B_{1/2}\cap K)} \le M,$$
with $M$ depending on $\|u\|_{L^\infty(B_1\cap K}$ and $G(1)$. Moreover, if $G(t)/t \in [\delta,\delta^{-1}]$ for $t \ge 1$, then the dependence of $M$ on $\|u\|_{L^\infty(B_1\cap K)}$ is linear and we have the estimate
$$ \|u\|_{C^{0,\alpha}(B_{1/2}\cap K)} \le C (1+ \|u \|_{L^\infty(B_1\cap K)}),$$
with $C$ depending only on $n$, $\lambda,\Lambda$, $\delta$.
\end{lem}

\begin{proof} The proof follows from Theorem 2.2 in \cite{DS}, or the arguments in Section 3. We sketch the argument for completeness. 

We will use the following version of the weak Harnack inequality, which is a straightforward consequence of Proposition \ref{prop:Harnack} part a):

\smallskip

Let $u \in \underline{\mathcal S}_\Lambda (K \cap B_1)$, $u \ge 0$, and $x_0 \in \overline K\cap B_{1/2}$ and $r\le \frac 12$. If
$$ \frac{|\{u=0\} \cap (B_{r}(x_0) \cap K)|}{|B_{r}(x_0) \cap K|} \ge \mu,$$
then
\begin{equation} \label{eqn:Holder1} \|u\|_{L^\infty(B_{r/2}(x_0) \cap K)} \le (1-c(\mu)) \|u\|_{L^\infty(B_{r}(x_0) \cap K)},\end{equation}
with $c(\mu)>0$ depending on the universal constants, $\delta$ and $\mu$. Moreover, we may take$c(\mu) \to 1$ as $\mu \to 1$.

\smallskip

 We want to show that $\|u\|_{C^{0,\alpha}(B_{1/2}\cap K)} \le M$, for some large constant $M$ depending on $\|u\|_{L^\infty(B_1\cap K)}$ and $G(1)$. 

Let $x_0 \in \overline K \cap B_{1/2}$ and let $r$ denote its distance to the free boundary. It suffices to show that 
\begin{equation}\label{600}|u(x_0)| \le M' r^\alpha,
\end{equation}
as then the $C^\alpha$ bound follows by applying the interior H\"older estimates (see Lemma \ref{cal}) in balls tangent to $F(u)$ that are included in the positive or the negative phase. 

Assume that $u(x_0)>0$. In order to prove \eqref{600}, we apply the weak Harnack inequality above for either $u^+$ or $u^-$ in each dyadic ball centered at $x_0$, depending on whether or not the density of $\{u \le 0\}$ is greater than $1-\mu$, where $\mu$ is chosen such that the constant in \eqref{eqn:Holder1} satisfies $c(\mu)= \frac 78$. If in half of the cases from the ball of radius $1/2$ up to the ball of radius comparable to $r$ we apply the Harnack inequality for $u^+$ then we obtain \eqref{600} for some small $\alpha$. Otherwise $u^-$ satisfies the inequality 
$$\|u^-\|_{L^\infty(B_{2r}(x_0) \cap K)} \le M_0 r^{3/2}.$$
Now we can use the free boundary condition to obtain the desired upper bound for $u(x_0)$. 
Assume by contradiction that $u(x_0) \ge M' r^\alpha$, which implies that $r$ is small as long as we take $M'$ large. 
Then the interior Harnack inequality implies that $u \ge c u(x_0)$ in $B_{r/2}(x_0)$.
 We compare $u$ in $(B_{2r}(x_0) \setminus B_{r/2}(x_0)) \cap \overline K$ with the function
 $$ a\psi^+ - b \psi^-, \quad \quad \psi:=  \gamma^{-1} r^{\gamma+1}\left(|x-x_0|^{-\gamma} - r^{-\gamma}\right).$$
 The function $\psi$ was chosen such that $\mathcal  M^-_\Lambda (D^2 \psi) >0$, and $|\nabla \psi|=1$ on $\p B_r(x_0)$.
  The constants $a$ and $b$ are chosen such that 
  $$u \ge a\psi \quad \mbox{ on} \quad  \p B_{r/2}(x_0), \quad \mbox{and} \quad u \ge - u^- \ge - b \psi^- \quad \mbox{ on} \quad  \p B_{2r}(x_0),$$
 and hence can be chosen to satisfy $a \sim M' r^{\alpha-1}$ and $ b \sim M_0 r^{1/2}.$ Then the maximum principle implies that $u \ge a\psi^+ - b \psi^-$ and the two functions touch at some point on $F(u) \cap \p B_r(x_0)$. This is a contradiction since $a\psi^+ - b \psi^-$ is an admissible test function as the Neumann boundary conditions are satisfied along $\p K$ while on $\p B_r(x_0)$ we have $a > G(1)>G(b)$ provided that $M'$ is chosen sufficiently large.

 Finally, we may use scaling to deduce the linear dependence of $M$ on $\|u\|_{L^\infty}$ if $G(t)/t \in [\delta,\delta^{-1}]$ for $t \ge 1$. Indeed, after dividing by $1+ \|u \|_{L^\infty(B_1\cap K)}$, we reduce to the case $\|u \|_{L^\infty(B_1\cap K)} \le 1$ and $G(1) \in [\delta, \delta^{-1}]$.

\end{proof}

Next, we prove a compactness lemma for sequences of solutions to \eqref{eqn:solution} with $L^{\infty}(K\cap B_1)$ norms tending to infinity.
\begin{lem}[Compactness] \label{lem:compactness} Assume that $u_k$ is a sequence of solutions to \eqref{eqn:solution} with
$$ \mathcal F_k^+=\mathcal F_k^-=\mathcal F_k, \quad \mbox{and} \quad G_k(t)/t \in [1-\eps_k, 1+ \eps_k]  \quad \mbox{for} \quad t\ge t_k.$$ 
If $\eps_k \to 0$ and $\|u_k\|_{L^\infty(K\cap B_1)}/t_k \to \infty$, then the rescaled functions
$$\tilde u_k:= \frac{u_k}{ \|u_k\|_{L^\infty(K\cap B_1)}},$$ 
converge (up to extracting a subsequence) uniformly locally to a solution $u^*$ to the Neumann problem  \begin{equation}\label{NP*}\mathcal{F}^*(D^2u^*) = 0 \quad \text{in $K\cap B_1$,} \quad  u^*_{\nu} = 0 \quad \text{on $\pa K\cap B_1$,}\end{equation} for some $\mathcal F^* \in \mathcal E(\lambda, \Lambda).$
\end{lem}
\begin{proof}
Setting $C_k = \|u_k\|_{L^\infty(K\cap B_1)}$, we define $$\tilde{\mathcal{F}}_k = \tfrac{1}{C_k}\mathcal{F}_k(C_k\cdot), \quad \quad  \tilde{{G}}_k(t) =\tfrac{1}{C_k}G_k(C_kt).$$ Then, $\tilde{u}_k$ solve the two-phase Neumann problem \eqref{eqn:solution} in the viscosity sense, with $\mathcal{F}^{\pm} = \tilde{\mathcal{F}}_k$ and $G = \tilde{G}_k$. After passing to a subsequence (and using the regularity from Lemma \ref{lem:Holder}) we have $\tilde{u}_k\to u^*$, $\tilde{\mathcal{F}}_k \to\mathcal{F}^*$, and $\tilde{{G}}_k\to G^*$, with $G^*(t) = t$, uniformly on compacts. Now the conclusion follows from Corollary \ref{cormain}. 
\end{proof}

A key step in the proof of Theorem \ref{thm:Lipschitz} is to show that if the $L^{\infty}(K\cap B_1)$ norm of a solution is sufficiently large, then either $u$ decays linearly as we approach a point on $F(u)\cap  \pa K$, or it is close to a two-plane solution.
\begin{lem} [Dichotomy] \label{lem:dichotomy} 
Assume we are in the setting of Theorem \ref{thm:Lipschitz}. Then, given $\eps>0$, there exist constants $r$, $\eps_0$ (see \eqref{eps0}) depending on $\eps$ and $n$, $\lambda, \Lambda$, $\delta$, and a large constant $M$ depending also on $t_0$ such that if  $\|u\|_{L^\infty(K\cap B_1)} \ge M$, then either

i) $$ \frac 1 r \|u\|_{L^\infty(K \cap B_r)} \le \frac 1 2 \|u\|_{L^\infty(K\cap B_1)},$$
or

ii) $$ \frac 1 r \| \tilde u- U_{b}\|_{L^\infty(K \cap B_r)}\leq \eps, \quad \quad  \tilde u=u/\|u\|_{L^\infty(K\cap B_1)}.
$$
Here $U_b$ denotes a two-plane solution $$U_{b}(x) := a (x\cdot \omega)^+ - b(x\cdot \omega)^-, \quad |\omega|=1, \quad \omega \in \{0\}\times \R^m,$$
for the rescaled problem \eqref{eqn:solution} that $\tilde u$ satisfies, and
$$a=\tilde G(b), \quad \quad b \in [\tfrac 14, C], \quad \quad \tilde G'(t) \in [1-\eps,1+\eps] \quad \mbox{for} \quad t \ge \tfrac 18,$$
with $C$ depending on $n$, $\lambda$, $\Lambda$, $\delta$, and with $\tilde{G}(t) = G(\|u\|_{L^{\infty}}t)/\|u\|_{L^{\infty}}$.

\end{lem}
\begin{proof}
We use an analogous proof to Proposition 2.4 in \cite{DS}. The proof is by compactness. Fix constants $r$, $C$ to be specified below. Suppose that there exist a sequence $\eps_k \to 0$ and solutions $u_k$, with $M_k = t_0 \eps_k^{-1}$, $\eps_0=\eps_k$ for which neither option holds. Applying Lemma \ref{lem:compactness}, we obtain a limit $u^*$ for the $\tilde u_k= u_k/\|u_k\|_{L^{\infty}}$, which solves the Neumann problem \eqref{NP*}.

Applying Theorem \ref{T08} to $u^*$, we obtain\begin{align} \label{eqn:dichotomy1}
|u^*(x) - \tau_m\cdot x| \leq C|x|^{1+\alpha},
\end{align}
in $K \cap B_1$, with $\tau_m\in \{0\}\times \R^m$, $|\tau_m|\leq C$. We now split into two cases, depending on the size of $|\tau_m|$. 

\smallskip

\textit{Case 1.} $|\tau_m| \leq \tfrac{1}{4}$.

Here, in $K\cap B_r$, we have $$\tfrac{1}{r}|u^*| \leq \tfrac{1}{4} + Cr^\alpha \leq \tfrac{1}{3},$$ by choosing $r$ sufficiently small. Therefore, the first option holds for $u_k$ with $k$ sufficiently large, giving a contradiction. (Note that if $m = 0$, then $\tau_m=0$, and so we are automatically in this case, and there is no dichotomy.)

\smallskip

\textit{Case 2.} $|\tau_m| > \tfrac{1}{4}$.

Then, $$\tfrac{1}{r}|u^* - x\cdot\tau_m| \leq Cr^{\gamma},$$ in $K\cap B_r$ and we choose $r$ sufficiently small so that $Cr^{\gamma} \leq \tfrac{1}{2}\eps$. 

We pick $\omega = \tau_m/|\tau_m|$, $b = |\tau_m|$ and $a = \tilde G_k(b)$, with
$$\tilde{G}_k(t) = G_k(\|u_k\|_{L^{\infty}}t)/\|u_k\|_{L^{\infty}}.$$ Using that $\tilde{G}_{k}(t)$ converges to the identity on compacts, the second option holds for $u_k$ for $k$ sufficiently large, again giving a contradiction.
\end{proof}

If $u$ satisfies the second estimate for a sufficiently small $\eps$, then it satisfies a similar conclusion as that of Theorem \ref{T09}. Precisely, the following estimate holds:

\begin{prop}\label{prop:improvement}
Suppose that $\tilde u$ satisfies the second alternative of Lemma \ref{lem:dichotomy} for a sufficiently small $\eps>0$ depending on $n$, $\lambda$, $\Lambda$, $\delta$. 

Then, there exist a unit vector $\omega_0 \in \{0\}\times \R^m$, and positive constants $a_0$, $b_0$ with $a_0 = \tilde G(b_0)$, $b_0 \in [\frac 18, 2C]$ such that
\begin{align*}
    |\tilde u(x) - (a_0(x\cdot \omega_0)^+ - b_0(x\cdot\omega_0)^-)| \leq C\eps |x|^{1+\alpha},
\end{align*}
in $K\cap B_{r}$. Here $C$, $\alpha$ are constants depending on $n$, $\lambda$, $\Lambda$, $\delta$. 
\end{prop}
 The conclusion implies that $|\tilde u(x)| \le C |x|$ in $K\cap B_{r}$, which gives 
 \begin{equation}\label{lasteq}
|u(x)| \le C \|u\|_{L^\infty(K \cap B_1)} |x| \quad \quad \forall \, x \in K \cap B_r,
\end{equation}
 and that $F(u)$ intersects $\p K$ orthogonally at the origin.
 
The proof of Proposition \ref{prop:improvement} follows from iterating the general {\it improvement of flatness} result given in Lemma \ref{impf} below.  
We first complete the proof of Theorem $\ref{thm:Lipschitz}$.

\begin{proof1}{Theorem \ref{thm:Lipschitz}}
The argument given in \cite{DS} continues to hold, with the first estimate in Lemma \ref{lem:dichotomy} and Proposition \ref{prop:improvement} above used in place of Proposition 2.4 from \cite{DS}. Precisely, let 
$$a(s):=s^{-1}\|u\|_{L^{\infty}(K\cap B_s)},$$
and for the Lipschitz regularity of $u$ at the origin we need to show that 
\begin{align} \label{eqn:Lipschitz1}
    a(r^k) \leq C\max\{\|u\|_{L^{\infty}(K\cap B_1)},M\}, \quad \quad \forall \, k \ge 0,
\end{align}
 with $r$ and $M$ as in Lemma \ref{lem:dichotomy}, and $C$ a large constant.
 
 We prove \eqref{eqn:Lipschitz1} by iterating Lemma \ref{lem:dichotomy}. For $C$ chosen greater than $1$, \eqref{eqn:Lipschitz1} holds for $k=0$. 

 If $a(r^k) \le M$ then we have $$a(r^{k+1}) \le C_0M,$$ for $C_0 = r^{-1}$. 
 
If $a(r^k) \ge M$, then we apply Lemma \ref{lem:dichotomy} to the rescaling $$u_k(x) = u(r^kx)/r^k.$$ Notice that the cone $K$ and the bounds on the function $G$ remain invariant under this rescaling. If the first alternative holds for $u_k$, then $$a(r^{k+1}) \le \frac 12 a(r^k),$$
and we continue with the iteration. 

Finally, if we end up in the second alternative for some value $k=k_0$, then we stop the iteration and apply \eqref{lasteq} to conclude
$$a(r^l) \le C_1 a(r^{k_0}), \quad \forall \, \, l \ge k_0,$$ 
for some $C_1$ large. 

Now \eqref{eqn:Lipschitz1} follows easily by choosing $C$ large depending on $C_0$ and $C_1$.
\end{proof1}

Proposition \ref{prop:improvement} follows from applying the following lemma indefinitely.
\begin{lem}[Improvement of flatness] \label{impf}
Let $u$ be a solution of \eqref{eqn:solution}, and assume that in a compact interval $I \subset (0,\infty)$, $G$ has the properties
$$ G(t), G'(t) \in [\delta, \delta^{-1}], \quad \quad |G'(t_1)-G'(t_2)| \le \eps_0 + \delta^{-1}|t_1-t_2|^\delta.$$
Suppose further that $0 \in F(u)$ and $u$ satisfies the flatness assumption
\begin{align*}
    U_b(x_n-\eps)  \leq u(x) \leq U_b(x_n+\eps) \quad\text{in }B_1\cap K, 
\end{align*}
with $U_b$ a two-plane solution with $(b-\delta, b + \delta) \subset I$.

There exist constants $\eps_0$, $r$, $C_0$ depending only on $n$, $\lambda$, $\Lambda$, $\delta$, $I$, such that if $\eps\leq \eps_0$ then
\begin{align*}
    U_{b'}(x\cdot \omega -r\eps/2)  \leq u(x) \leq U_{b'}(x\cdot \omega+r\eps/2) \quad\text{in }B_r\cap K,
\end{align*}
with $\omega \in \{0\}\times \R^m$, $|\omega|=1$, $|\omega-e_n|<C_0\eps$, $|b'-b|<C_0 \eps$.
\end{lem}
\begin{proof}
    This lemma is the extension of the non-degenerate improvement of flatness Lemma 5.1 in \cite{DFS1} to the boundary of the cone $K$. The same strategy of proof works here. For completeness, we sketch here the main changes for our setting.

    The proof proceeds by contradiction, by assuming a sequence $\eps_k \le \eps_{0,k}\to 0$, solutions $u_k$, operators $\mathcal{F}_k\in \mathcal{E}(\lambda,\Lambda)$, functions $G_k$, and constants $b_k$ satisfying the hypotheses of the lemma but not the conclusion. Notice that the properties of $G$ imply that for any fixed $t$,
    \begin{equation}\label{gex}
    G_k( b_k+ \eps_k t) = G_k(b_k) + G_k'(b_k) \eps_k t + o(\eps_k).
    \end{equation}

\smallskip

    \textit{Step 1.} The first step is a compactness argument. Setting $a_k = G_k(b_k)$, we define $\tilde{u}_k$ by
    \begin{equation}\label{tilde}
    \tilde{u}_k(x) = 
    \begin{cases}
        \displaystyle{\frac{u_k(x)-a_kx_n}{a_k \eps_k}},& \quad x\in B_1^+(u_k)\cup F(u_k) \\
        
        \ \\
        
         \displaystyle{\frac{u_k(x)-b_kx_n}{b_k\eps_k}},& \quad x\in B_1^-(u_k).
    \end{cases}
    \end{equation}
    The flatness hypothesis then ensures that $-1\leq \tilde{u}_k\leq 1$ in $B_1\cap K$, and that $F(u_k)$ converges to $B_1\cap\{x_n=0\}\cap K$. As in Lemma 5.1 in \cite{DFS1}, we want to conclude that (up to a subsequence) 
    $$\tilde{u}_k \to  \tilde{u} \quad \mbox{uniformly in} \quad \overline K \cap B_{1/2}.$$

     In \cite{DFS1} this follows from a non-degenerate Harnack inequality, see Theorem 4.1 in \cite{DFS1}. The same proof carries through in our setting. It uses the classical Harnack inequality which by Proposition \ref{prop:Harnack} holds in the Neumann boundary condition case as well, and barriers function of the type 
     $$a\psi^+ - b \psi^- \quad \mbox{ with} \quad \psi=x_n \pm \eps w,$$ with $w$ defined in the annulus $A = B_{3/4}(\bar{x})\backslash \bar{B}_{1/20}(\bar{x})$ by,
    \begin{align*}
        w(x) = c(|x-\bar{x}|^{-\gamma} - (3/4)^{-\gamma}), \quad \quad \bar x \in K \cap B_{1/4}^+.
    \end{align*}
       The constant $\gamma$ is chosen suitably large so that $\mathcal{M}^{-}_{\frac \lambda n, \Lambda}(D^2w)>0$. 
       
       As shown in the previous sections, these barriers remain suitable for the Neumann boundary condition on $\p K$. This is because the cone $K$ is invariant in the $e_n$ direction and $\nabla w(x)$ has the direction of $\bar{x}-x$, hence $$\nabla \psi(x)=e_n \pm \eps \nabla w(x), \quad \quad x \in \p K,$$ points inside/outside the tangent cone $K_x$. Then the proof of Theorem 4.1 in \cite{DFS1} applies in our setting.

  From our hypotheses we may also assume (after passing to a subsequence) the convergence of the operators on compact sets
    $$\mathcal{F}_k^{+} = \frac{1}{a_k\eps_k}\mathcal{F}_k(a_k\eps_k\cdot) \to \tilde {\mathcal F}^+, \quad \quad \mathcal{F}_k^{-} = \frac{1}{b_k\eps_k}\mathcal{F}_k(b_k\eps_k\cdot) \to \tilde {\mathcal F}^-,$$
    and
    $$G_k \to G, \quad G_k'\to G' \quad \mbox{uniformly in $I$}, \quad \quad b_k \to b, \quad a_k \to a=G(b).$$
   
\smallskip

    \textit{Step 2.} The next step is to show that the limiting solution $\tilde{u}$ from Step 1 satisfies the transmission problem studied in Section 5,
    $$
    \begin{cases}
             \mathcal{F}^{\pm}(D^2\tilde{u}) = 0 &\text{ in } B^{\pm}_{1/2}\cap K\cap \{x_n\neq0\} \\
             
             \  \\
             
             a \, \tilde{u}_n^+ + b G'(b)\, \tilde{u}_n^- = 0 &\text{ on } B_{1/2}\cap K \cap\{x_n=0\},
    \end{cases}$$
    together with the Neumann boundary conditions on $\pa K$.
    
    The fact that $\tilde u$ satisfies the transmission problem in the interior of $K$ was established in Lemma 5.1 in \cite{DFS1}. It remains to check the Neumann boundary conditions on $\p K$. On $\p K \cap \{x_n \ne 0\}$ this is a direct consequence of the stability property Proposition \ref{stb}. In view of Definition \ref{def5.1}, we establish the property on $\p K \cap \{x_n = 0\}$ if we show that $\tilde u$ cannot be touched by below say at $0$ by a piecewise quadratic polynomial $P$ of the form
\begin{equation}\label{Pq}
    P(x) = A+px_n^+ + qx_n^- + B  \cdot Q (x-y),
\end{equation}
with $A \in \R$, 
$$Q(x) := \frac{1}{2}(\gamma x_n^2 - |x'|^2), \quad y = (y',0), \quad y' \in K',$$ and 
\begin{equation}\label{gex2}
\mathcal M^-_\Lambda(D^2 Q)>0, \quad B>0, \quad a \, p+ b G'(b) \, q>0.
\end{equation}
   Let us assume by contradiction that $\tilde u$ is touched strictly by below at $0$ by such a quadratic polynomial $P$.

    The proof follows precisely as in the interior of $K$ in Lemma 5.1 in \cite {DFS1}. For completion we provide some of the details. 
    
    Denote by
    $$\Gamma (x)= \frac{1}{\gamma-1} \left(( |x'|^2 + (x_n-1)^2 )^{-\frac{\gamma-1}{2}} -1 \right), $$
  the radial subsolution with axis at $e_n$, which has the property that 
   $$\Gamma(x) = x_n + Q (x) + O(|x|^3),$$
   and
   $$ \mathcal M^-_{\frac \lambda n, \Lambda}(D^2 \Gamma) >0, \quad \{ \Gamma=0\}= \p B_1(e_n), \quad |\nabla \Gamma|=1 \quad \mbox{on} \quad \{\Gamma=0\}.$$
    Let $\Gamma_k$ be the following dilation and translation of $\Gamma$
    $$ \Gamma_k(x)= \frac{1}{B \eps_k} \Gamma (B \eps_k (x -y + A \eps_k e_n)),$$
    and consider the test function
    $$\phi_k := a(1+p \eps_k) \cdot \Gamma_k^+ - b(1- q \eps_k) \cdot \Gamma_k^-.$$
    We remark that $\phi_k$ is an admissible test function for the subsolution property since 
    
    1) $\nabla \Gamma_k(x)$ points in the direction of $\bar y - x$ with
    $$ \bar y= y + (\tfrac{1}{B \eps _k}-A \eps_k) e_n \quad \in K ,$$
    
    2) $\mathcal M^-_{\frac \lambda n, \Lambda}(D^2 \Gamma_k) >0$;
    
    3) $|\nabla \Gamma_k|=1$ on $\Gamma_k=0$ hence, by using \eqref{gex} and \eqref{gex2},
    $$|\nabla \phi_k^+| =a(1+p \eps_k) > G_k(b(1-q \eps_k)) = G_k(|\nabla \phi_k^-|),$$
    for all large $k$.
    
    Since
    $$\Gamma_k(x)= x_n + \eps_k (A + B \cdot Q(x-y)) + O(\eps_k^2),$$
     when we compute the corresponding function $\tilde \phi_k$ associated with $\phi_k$ (see \eqref{tilde}) we find
    $$ \tilde \phi_k = P + O(\eps_k),$$
    with $P$ as in \eqref{Pq}. This means that we can add a small constant to $\tilde \phi_k$ so that the new function touches $\tilde u_k$ by below at some point close to $0$. In turn, this means that a small translation of $\phi_k$ touches $u_k$ by below and we reach a contradiction for $k$ sufficiently large.      
     \smallskip

\textit{Step 3.} The final step in the proof of this lemma is to use the regularity properties of the limiting function $\tilde{u}$ to obtain a contradiction. Since $\tilde{u}(0) = 0$ and $|\tilde{u}|\leq 1$ in $K \cap B_1$, from Theorem \ref{T09}, we have
\begin{align*}
    |\tilde{u}(x) - (p x_n^+ + q x_n^-+\tau_m\cdot x)| \leq C|x|^{1+\beta},
\end{align*}
in $B_{\delta}\cap K$, and
$$|p|, |q|, |\tau_m| \le C, \quad ap + G'(b)b q =0, \quad \tau_m \in \{0\}\times \R^m.$$
Here $\tilde{p}$ and $\tilde{q}$ satisfy $\tilde{\alpha}G'(\tilde{\alpha}) \tilde{p} = \tilde{\beta}\tilde{q}$, and $\tau_m\cdot e_j=0$ for $j=n$ and $j\leq n-m$. Since by Step 1, $\tilde{u}_k$ converges uniformly to $\tilde{u}$, we find
$$ |\tilde u_k -  (p x_n^+ + q x_n^-+\tau_m\cdot x)| \leq \tfrac 1 4 r, \quad \mbox{in} \quad \overline K \cap B_r,$$
with $r$ small, universal.
 Setting
\begin{align*}
    b_k' = b_k(1-\eps_k q), \quad \omega_k = \frac{1}{\sqrt{1+\eps_k^2|\tau_m|^2}}(e_n + \eps_k\tau_m),
\end{align*}
it is not difficult to check that (see the final step of Lemma 5.1 in \cite{DFS1}) 
\begin{align*}
U_{b_k'}(x\cdot \omega_k- \tfrac {\eps_k}{ 2} \, r) \leq u_k(x) \leq U_{b_k'}(x\cdot \omega_k+ \tfrac{\eps_k}{2} r) \quad \quad  \text{ in }B_{r}\cap K,
\end{align*}
for sufficiently large $k$, giving a contradiction and completing the proof.
\end{proof}

\end{document}